\documentclass[11pt,a4paper]{article}

\usepackage[left=2.0cm, top=2.45cm,bottom=2.45cm,right=2.0cm]{geometry}
\usepackage{mathtools,amssymb,amsthm,mathrsfs,calc,graphicx,xcolor,cleveref,dsfont,tikz,pgfplots,bm, bbm}
\usepackage{tikz-cd}

\usepackage[british]{babel}
\usepackage{amsfonts}              % for blackboard bold, etc
\usepackage[T1]{fontenc}

\numberwithin{equation}{section}

\newcommand{\ind}{\mathbbm{1}}

\newtheorem {theorem}{Theorem}[section]

\newtheorem {lemma}[theorem]{Lemma}
\newtheorem {corollary}[theorem]{Corollary}

{\theoremstyle{definition}

}

{
\theoremstyle{remark}
\newtheorem {remark}[theorem]{Remark}
}

\def\EE{\mathbb{E}}

\def\NN{\mathbb{N}}

\def\PP{\mathbb{P}}

\def\RR{\mathbb{R}}
\def\RRd1{\mathbb{R}^{d+1}}

\def\cA{\mathcal{A}}
\def\cB{\mathcal{B}}

\def\cE{\mathcal{E}}
\def\cF{\mathcal{F}}

\def\area{\textup{area}}
\newcommand{\var}{\operatorname{Var}}
\newcommand{\cov}{\operatorname{Cov}}

\newcommand{\eqdistr}{\stackrel{d}{=}}

\makeatletter
\let\@fnsymbol\@alph
\makeatother

\begin{document}

\title{\bfseries Variance expansion and Berry-Esseen bound for the number of vertices of a random polygon in a polygon}

\author{Anna Gusakova\footnotemark[1]\;\;\;\;\;\;  Matthias Reitzner\footnotemark[2] \;\;\;\;\;\; Christoph Th\"ale\footnotemark[3]}
\date{}
\renewcommand{\thefootnote}{\fnsymbol{footnote}}
\footnotetext[1]{Fachbereich Mathematik, Universit\"at M\"unster, Germany. Email: gusakova@uni-muenster.de}
\footnotetext[2]{Institut f\"ur Mathematik, Universit\"at Osnabr\"uck, Germany. Email: matthias.reitzner@uos.de}
\footnotetext[3]{Fakult\"at f\"ur Mathematik, Ruhr-Universit\"at Bochum, Germany. Email: christoph.thaele@rub.de}

\maketitle

\begin{abstract}
	Fix a container polygon $P$ in the plane and consider the  convex hull $P_n$ of $n\geq 3$ independent and uniformly distributed in $P$ random points. In the focus of this paper is the vertex number of the random polygon $P_n$. The precise variance expansion for the vertex number is determined up to the constant-order term, a result which can be considered as a second-order analogue of the classical expansion for the expectation of R\'enyi and Sulanke (1963). Moreover, a sharp Berry-Esseen bound is derived for the vertex number of the random polygon $P_n$, which is of the same order as the square-root of the variance. The main idea behind the proof of both results is a decomposition of the boundary of the random polygon $P_n$ into random convex chains and a careful merging of the variance expansions and Berry-Esseen bounds for the vertex numbers of the individual chains.

\bigskip

\end{abstract}

\section{Introduction and result}

Let $P\subset \RR^2$ be polygon in the plane with $\ell\ge 3$ vertices. We refer to $P$ as a container in what follows. Let $X_1,\ldots, X_n$ be $n\geq 3$ independent random points, distributed uniformly in $P$. We denote by $P_n$ the random convex hull $[X_1,\ldots, X_n]$ of these points. It is a random polygon in the container polygon $P$, see Figure \ref{fig:Polygon}. In this article we are interested in the combinatorial structure of $P_n$, more precisely in the variance expansion and the fluctuations of the number $f_0(P_n)$ of vertices of $P_n$. Note that this quantity is the same as the number $f_1(P_n)$ of edges of $P_n$. 

\begin{figure}[t]
\centering
\includegraphics[width=0.5\columnwidth]{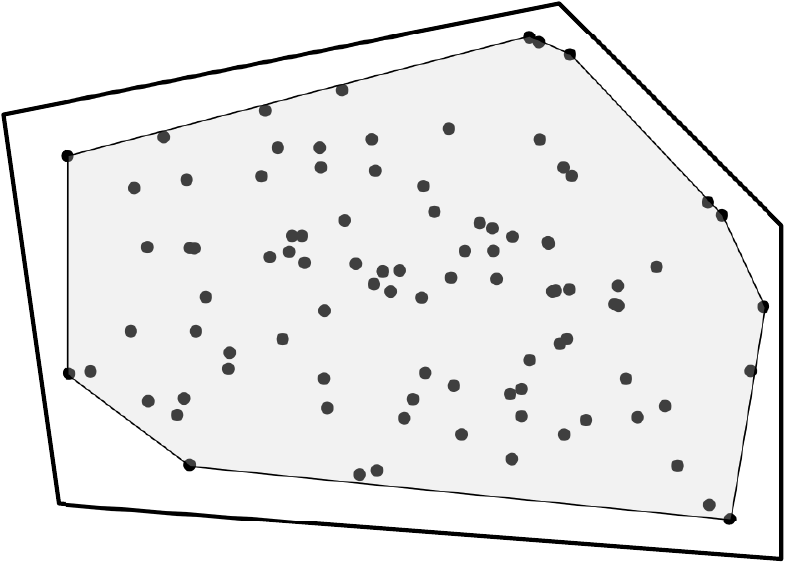}
\caption{Random polygon in a polygon.}
\label{fig:Polygon}
\end{figure}

The random variable $f_0(P_n)$ has been intensively studied in the literature and has attracted a lot of interest in geometric probability as well as convex and integral geometry. For example, as $n\to\infty$, R\'enyi and Sulanke in their fundamental article \cite{RS63} have established the asymptotics
\begin{equation}\label{eq:expPolUnif}
\EE f_0(P_n)=
{2 \ell \over 3}\log n + \frac 23 \sum_{i=1}^\ell \log \left(\frac{F_i}{\area(P)}\right)  + \frac{2 \gamma \ell }3  + o(1),
\end{equation}
for the expected number of vertices of $P_n$, where $F_i$, $i=1, \dots, \ell$, are the areas of the triangles formed by three consecutive vertices of $P$ (that is, $F_i=\area([v_{i-1},v_i,v_{i+1}])$ if $v_1,\ldots,v_\ell$ are the vertices of $P$ with the convention that $v_0=v_\ell$ and $v_{\ell+1}=v_1$). The corresponding variance asymptotics
\begin{equation}\label{eq:varPolUnif}
\var f_0(P_n)={10\ell\over 27}\log n (1+o(1)),
\end{equation}
is due to Groeneboom \cite{Gro88}. In the same paper, Groeneboom also proved a central limit theorem for $f_0(P_n)$, as $n\to\infty$.
Denoting by $\Phi(\,\cdot\,)$ the distribution function of a standard Gaussian random variable, B\'ar\'any and Reitzner have obtained the following quantitative version of the central limit theorem:
\begin{equation}\label{eq:BaraReitz}
\sup_{x\in\RR}\Big|\PP\Big({f_0(P_n)-\EE f_0(P_n)\over \sqrt{\var(P_n)}}\leq x\Big)-\Phi(x)\Big|\leq c\,{(\log\log n)^{60}\over \sqrt{\log n}},
\end{equation}
where $c>0$ is some constant not depending on $n$ (note that in the published version \cite{BR10} of \cite{BR06} a different model for $P_n$ has been treated, which involves an additional randomization of the number of generating random points, which will also play a prominent role in our analysis). 
The double logarithmic factor also appears in a central limit theorem of Pardon \cite{Pardon1, Pardon2}, which holds for general planar convex container sets and is not restricted to the polygons or convex sets bounded by a sufficiently smooth closed curve.
The purpose of the present article is to demonstrate how the double-logarithmic factor in \eqref{eq:BaraReitz} can be removed for polygons. Since the resulting rate of convergence is then of the same order as $(\var f_0(P_n))^{-1/2}$, we believe that, up to numerical constants, our result is in fact optimal. Moreover, our technique allows us at the same time to determine the precisely expansion of \eqref{eq:varPolUnif} up to the constant-order term, leading to a second-order analogue of the classical result \eqref{eq:expPolUnif} for the expectation of R\'enyi and Sulanke. Even further, our technique allows us to make more precise the $o(1)$-term in \eqref{eq:expPolUnif}.

\begin{theorem}\label{thm:main}
For every planar polygon $P\in\RR^2$ and $n\geq 3$ we have that
$$
\sup_{x\in\RR}\Big|\PP\Big({f_0(P_n)-\EE f_0(P_n)\over \sqrt{\var f_0(P_n)}}\leq x\Big)-\Phi(x)\Big|\leq {c\over \sqrt{\log n}},
$$
where $c>0$ is some constant independent of $n$, with 
\begin{align*}
\EE f_0(P_{n}) &=
{2 \ell \over 3}\log n + \frac 23 \sum_{i=1}^\ell \log \left(\frac{F_i}{\area(P)}\right)  + \frac{2 \gamma \ell }3  + O((\log n)^2n^{-1/4}) 
\intertext{and}
\var f_0(P_n) 
&=
{10 \ell \over 27}\log n +{10\over 27}\sum_{i=1}^\ell \log \left(\frac{F_i}{\area(P)}\right) + 
\frac{(10 \gamma -2 \pi^2 )\ell }{27} 
+ O( (\log n)^4  n^{-1/4}),
\end{align*}
as $n\to\infty$, where $\gamma\approx 0.57721\ldots$ is the Euler-Mascheroni constant.
\end{theorem}

\begin{remark}
	On the way of the proof of Theorem \ref{thm:main} we obtain in Theorem \ref{thm:BerryEsseenPoisson} the same statements also for the Poisson random polytope model previously considered in \cite{BR10}, that is, under the additional randomization that the number of generating points follows a Poisson distribution with mean $n$. 
\end{remark}

\begin{remark}
Although this paper focusses on the vertex number of the random polygon $P_n$, we briefly discuss a direct consequence of Theorem \ref{thm:main} on the area of $P_n$. First, the classical Efron identity \cite{Ef} connects the expected number of vertices of $P_n$ with its area: 
$$
{\EE\area(P_n)\over\area(P)} = 1- {\EE f_0(P_{n+1})\over n+1}.
$$
Applying the expansion for $\EE f_0(P_{n+1})$ in Theorem \ref{thm:main} we conclude that, as $n\to\infty$,
$$
\frac{\EE \area(P \setminus P_n)}{\area(P)} 
=
{2 \ell \over 3} (\log n)n^{-1}+ \bigg[\frac {2}3 \sum_{i=1}^\ell \log \left(\frac{F_i}{\area(P)}\right)  + \frac{2 \gamma \ell }3 \bigg] n^{-1}  + O( (\log n)^2 n^{-5/4}) .
$$
Similarly, one can apply Buchta's identity \cite[Corollary 1]{Bu11}, which connects the variance of the area of $P_n$ with the first two moments of its vertex number:
$$
{\var\area(P_n)\over\area(P)^2} = {\var f_0(P_{n+2})+A_n-B_n\over(n+1)(n+2)}
$$
with $A_n:=(\EE f_0(P_{n+2}))^2-{n+2\over n+1}(\EE f_0(P_{n+1}))^2$ and $B_n:=(2n+3)\EE f_0(P_{n+2})-2(n+2)\EE f_0(P_{n+1})$. In connection with the variance expansion in Theorem \ref{thm:main} this leads to
\begin{align*}
\frac {\var \area(P_n)}{\area(P)^2}=
{28 \ell \over 27} (\log n)n^{-2} +\bigg[{28\over 27}\sum_{i=1}^\ell \log \left(\frac{F_i}{\area(P)}\right) + 
\frac{(28 \gamma -2 \pi^2 )\ell }{27}\bigg]n^{-2}
 + O( (\log n)^4  n^{-9/4}) ,
\end{align*}
as $n\to\infty$.
\end{remark}

We would like to point out that if the container set $P$ of the random polygon $P_n$ has a $C^2$-smooth boundary with everywhere positive curvature the first Berry-Esseen bound for the vertex number in \cite{Reitz05} also contained an additional logarithmic factor, and the presumably optimal Berry-Esseen bound has been found in \cite[Theorem 3.5]{LachSchulteYukich}. Moreover, the approach in \cite{LachSchulteYukich} even allows to deal with higher dimensional random polytopes and with other geometric and combinatorial parameters. However, we would like to stress at this point that the transition from smooth container sets to polygons appears to be highly non-trivial. One reason for this fact is the observation that, in contrast to smooth containers, the geometry of a random polygon in a polygon is not locally determined in a sufficiently strong sense. For example, for an arbitrary boundary point $x$ of the container set one can ask for the expected number of vertices of the random polygon $P_n$ one can "see" from $x$. While in the smooth case this number stays bounded for large $n$,  in case of a polygon it grows to infinity at a double logarithmic speed. These and similar geometric facts are the reason why already the proof of a sub-optimal Berry-Esseen bound like \eqref{eq:BaraReitz} is considerably more involved compared to its counterpart for smooth container sets. Another aspect to be mentioned in this context is the strong concentration of the number of vertices of a random polygon in a small neighbourhood around the corners of the container polygon. This concentration phenomenon, which apparently does not take place in smoothly bounded container sets, makes it much harder to approximate the random polygon by the so-called floating body associated with the container. More precisely, even the careful refinement from \cite{BR10} of this approach automatically leads to the double-logarithmic factor in \eqref{eq:BaraReitz}. 

We shall now briefly explain how we overcome these difficulties. In particular, this description makes it evident that our approach cannot be extended to deal with the vertex number, or more generally the number of faces of arbitrary dimension, of convex hulls of random points in polytopes of dimension more than two. The main ingredient we use in the proof of Theorem \ref{thm:main} is a decomposition of the boundary of the random polygon $P_n$ into so-called random convex chains, where each chain corresponds to one of the corners of the container polygon, see Figure \ref{fig:Step2}. In fact, the vertex number of $P_n$ is the same as the sum of the number of vertices of these chains, which after Poissonization of the construction become independent random variables. By a suitable affine transformation, each chain can be transformed into the following standard form without changing its combinatorial structure (of course, the number $m$ below depends in a suitable way on the size of the corresponding chain in $P_n$). Consider the triangle $T$ with vertices $(0,0)$, $(0,1)$ and $(1,0)$ and let $T_m$ be the convex hull of $(0,1)$ and $(1,0)$ together with $m\geq 1$ independent and uniformly distributed random points in $T$. The following Berry-Esseen bound for the vertex number $f_0(T_m)$ of the random convex chain $T_m$ has been obtained in \cite[Corollary 9]{GT21}:
\begin{equation}\label{eq:CLTChain}
\sup_{x\in\RR}\Big|\PP\Big({f_0(T_m)-\EE f_0(T_m)\over \sqrt{\var f_0(T_m)}}\leq x\Big)-\Phi(x)\Big|\leq {c\over\sqrt{\log m}},
\end{equation}
for some absolute constant $c>0$ and $m\geq 1$. We would like to mention at this occasion that this result is a consequence of an unexpected connection (which does not persist for polygons different from a triangle) between the random variables $f_0(T_m)$, the location of zeros of certain orthogonal polynomials related to probability generating function of $f_0(T_m)$ and a Berry-Esseen bound for sums of independent Bernoulli random variables. We remark that \eqref{eq:CLTChain} is the main motivation behind Theorem \ref{thm:main} and the highly non-trivial transition from \eqref{eq:CLTChain} to the Berry-Esseen bound for $f_0(P_n)$ in Theorem \ref{thm:main} is our main contribution. In particular, it involves a careful merge of the Berry-Esseen bounds of the individual convex chains based on Poissonization and conditioning arguments. A similar strategy will be applied in order to determine the asymptotic behaviour of the expectation  and variance of $f_0(P_n)$ in Theorem, \ref{thm:main}.

\section{Preliminaries}\label{sec:Prelim}

In this paper, $[A]$ stands for the convex hull of a set $A\subseteq\RR^2$ and $\# A$ for its cardinality. In addition, if $A$ is a line segment, we denote  its length by $|A|$. We write $B(x,r)$ for the closed disc centred at $x\in\RR^2$ with radius $r>0$. Given a set $A\subset \RR^2$ and a point $x\in\RR^2$ denote by $d(A,x):=\inf_{y\in A}\|y-x\|$ the distance between $A$ and $x$. For two functions $f$ and $g$ we write $f=O(g)$ if $\limsup\limits_{x\to\infty}|f(x)/g(x)|<\infty$ and $f=o(g)$ if $\lim\limits_{x\to\infty}|f(x)/g(x)|=0$. Given a line $L:=\{(x,y)\in\RR^2\colon ux+vy=t\}$ we denote by 
$$
L^+:=\{(x,y)\in\RR^2\colon ux+vy\ge t\}\qquad \text{and} \qquad L^-:=\{(x,y)\in\RR^2\colon ux+vy\leq t\}
$$ 
the positive and negative half-planes into which $L$ divides $\RR^2$, respectively. Given a convex body $K\in\RR^2$ we consider a function $v:K\mapsto\RR$, defined as
$$
v(z):=\min\{\area(H\cap K)\colon H\text{ is a half-plane with }z\in H\}.
$$
Then the floating body $K(v\ge \delta)$ with parameter $\delta>0$ is a level set of the function $v$, namely $K(v\ge \delta):=\{z\in K\colon v(z)\ge \delta\}$. The wet part is the set $K(v<\delta):=K\setminus K(v\ge \delta)$. In case of the polygon $P$ with $\area(P)=1$ the following formula for the area of the wet part (and an analogous formula for volume in arbitrary dimensions) was independently obtained in \cite{Schu91} and \cite{BB93}:
\begin{equation}\label{eq:wetpart}
	\area(P(v<\delta))={\ell\over 4}\delta  \log \frac 1{\delta}\, (1+o(1)).
\end{equation}

It was first observed by B\'ar\'any and Larman \cite{BarLar} that the random polygon $P_n$ is close to the floating body $P( v \geq n^{-1})$, with similar result for the Poissonized random  polytope. This connection was made precise in several aspects. 
B\'ar\'any and Dalla \cite[Theorem 1]{BD} proved the following fact, see also B\'ar\'any \cite[Theorem 7.4]{Barsurvey}. 
Note that it is possible to choose the same constant $b_0>0$ in the following two results in which we denote by $\overline{\cA}$ the complement of an event $\cA$.

\begin{lemma}\label{le:floating-Pn}
Choose $n$ independent uniform random points $X_1, \dots, X_n$ in a container polygon $P$. Then there is a constant $b_0>0$ such that the event $\cA_n := \{ P(v\ge b_0 n^{-1} \log n )\subset P_{n}\}$ satisfies
\begin{equation}
\PP(\overline{\cA_n} ) =O(n^{-6}).
\end{equation} 
\end{lemma}

We also have to deal with the Poissonized random polytope in the following. To define this model properly, let $N$ be a Poisson random variable with mean $n$, and choose $N$ independent uniform random points $X_1, \dots, X_N$ in a polygon $P$, which are also independent of $N$. Then $X_1, \dots, X_N$ is a homogeneous Poisson point process $\eta$ in the polygon $P$ with $\EE(\#\eta)=n$. We denote its convex hull by $P_\eta=[\eta]$.
The next lemma is a combination of Lemma 5.3 in \cite{BR10} and the Poissonized version of \cite[Theorem 1]{BD}. 

\begin{lemma}\label{le:floating-Peta}
There are constants $b_0, c_0>0$ independent of $n$, such that the events 
$$\cA^\pi_n := \{P(v\ge b_0 n^{-1} \log n)\subset P_\eta\}\qquad\text{and}\qquad
\cB^\pi_n := \{\#(\eta \cap P(v< b_0 n^{-1} \log n ))\leq c_0 (\log n)^2\}
$$
satisfy
$$
\PP(\overline{\cA^\pi_n}) \leq \PP(\overline{\cA^\pi_n  \cap \cB^\pi_n}) = O(n^{-6}).
$$
\end{lemma}

During the proofs of this paper we switch several times between the Poisson model $P_\eta$ and the binomial model $P_n$ of the random polygons we consider. For this the following estimate will turn out to be helpful. It is a slight extension of a result of Vervaat \cite{Vervaat}.

\begin{lemma}\label{le:diff-Poisson-binom}
Let $n \in \NN$ and $p>0$. Put ${n \choose m}=0$ for $n<m$. Then
\begin{equation*}
 \sum_{m=0}^\infty m^k \left|\frac{(np)^m}{m!} e^{-np} - {n \choose m} p^m (1-p)^{n-m}  \right| \leq 
\begin{cases}
2p & k=0, \\
2np^2 & k=1, \\
2np^2 (1+ n p) & k=2.
\end{cases}
\end{equation*}
\end{lemma}
\begin{proof}
The inequality for $k=0$ is due to Vervaat \cite{Vervaat}. The case $k=1$ follows from
\begin{align*}
\sum_{m=0}^\infty  m \left|\frac{(np)^m}{m!} e^{-np} - {n \choose m} p^m (1-p)^{n-m}  \right| 
& \leq 
np \sum_{m=0}^\infty \left|\frac{(np)^{m}}{m!} e^{-np} - {n-1 \choose m} p^{m} (1-p)^{n-m-1}  \right|  
\\ &\leq 2np^2. 
\end{align*}
The case $k=2$ is a combination of case $k=1$ with
\begin{align*}
\sum_{m=0}^\infty  m(m-1) 
&
\left|\frac{(np)^m}{m!} e^{-np} - {n \choose m} p^m (1-p)^{n-m}  \right| 
\\ \leq &
(np)^2 \sum_{m=0}^\infty  \left|\frac{(np)^{m}}{m!} e^{-np} - \frac{n-1}{n} {n-2 \choose m} p^{m} (1-p)^{n-m-2}  \right| 
\\ \leq & 
(np)^2 \sum_{m=0}^\infty  \left|\frac{(np)^{m}}{m!} e^{-np} -  {n-2 \choose m} p^{m} (1-p)^{n-m-2}  \right| 
\\ & + 
np^2 \sum_{m=0}^\infty  \left| {n-2 \choose m} p^{m} (1-p)^{n-m-2}  \right| 
\\  \leq & 
2 p (np)^2 + np^2 .
\end{align*}
This completes the argument.
\end{proof}

\section{Proof of Theorem \ref{thm:main}}\label{sec:ProofSquare}

The strategy of the proof of Theorem \ref{thm:main} consists of the following main steps:
\begin{description}
\item[{Step 1:}] As a first step we randomize the model further by taking instead of a fixed number $n$ of points, distributed uniformly in $P$, a random number $N$ of points. More precisely, as the number of points we use a Poisson random variable with mean $n$ so that the collection of random points becomes a homogeneous Poisson point process. This construction is known as Poissonization appears to be very helpful for our purposes. In fact, after Poissonization the number of points in disjoint regions become independent random variables.

\item[{Step 2:}] In order to prove the variance asymptotic and the Berry-Esseen bound for the Poisson model described in Step 1 we need to introduce an additional construction, which allows us to exclude some "bad" and "unlikely" events on which we do not have sufficient control on the geometric configuration.

\item[{Step 3:}] The third step is devoted to the proof of the variance asymptotic and the Berry-Esseen bound for the Poisson random convex chain. 
The corresponding results for the classical convex chain, proven in \cite{GT21}, are transformed using so-called "transfer lemma", which has been used in similar situations before in the literature, see \cite{BR06,BR10, BV06,Reitz05,Vu06}.

\item[{Step 4:}] The fourth step is concerned with the proof of the variance asymptotic and the Berry-Esseen bound for the Poisson random polygon in $P$ under the conditions introduced in Step 2. As already explained in the introduction, the idea of the proof is to decompose the boundary of the random polygon into independent random convex chains, each of which corresponds to one corner of container polygon. The result can now be obtained from the corresponding result for the Poissonized convex chain from Step 3.

\item[{Step 5:}] In this step we remove the constraints introduced in Step 2 and show that this does not change the quality of the variance asymptotic and the Berry-Esseen bound for the Poisson model.

\item[{Step 6:}] In the last step we deduce the variance asymptotic and the Berry-Esseen bound for the original model from the one for the Poisson model via de-Poissonization. The transition from the Poisson model to the binomial one is made using again the "transfer lemma".
\end{description}

In order to simplify some of our arguments we can and will from now on assume that the container polygon $P$ has area one. This is indeed possible as the vertex number of $P_n$ does not change under rescaling of the container polygon $P$.

\subsection{Step 1: Introducing the Poisson model} \label{subsec1}

Let $\eta$ be a homogeneous Poisson point process in the container polygon $P$ with $\EE(\#\eta) = n$.
Formally, $\eta$ can be constructed as follows. Let $N$ be a Poisson random variable with mean $n$ and, independently of $N$, let $X_1,X_2,\ldots$ be a sequence of independent and uniformly distributed random points in $P$. Then $\eta$ can be defined as the random set of points $\{X_1,\ldots,X_N\}$, which is interpreted as the empty set if $N=0$, an event having probability $e^{-n}$. We will refer to $\eta$ as a homogeneous Poisson point process with intensity $n$ on $P$. 
The well-known multivariate Mecke formula is a useful tool to compute expectations of Poisson functionals, see~\cite[Theorem 4.1]{LP}. It says that 
\begin{equation}\label{eq:Mecke}
\EE \sum_{(X_1,\ldots,X_k)\in \eta^k_{\neq}} f(X_1,\ldots,X_k,\eta) 
= 
n^k \, \EE \int\limits_{P^k} 
f(x_1,\ldots,x_k,\eta\cup\{x_1,\ldots,x_k\})\,d x_1\ldots d x_k
\end{equation}
where $\eta^k_{\neq}$ denotes the set of all $k$-tuples of distinct points of $\eta$, and $f$ is a non-negative measurable function acting on a $k$-tuple of points and a locally-finite point configuration in $\RR^2$ to which these points belong.

We consider now the random polygon $P_{\eta}$, which (we recall) was defined as a convex hull of the Poisson point process $\eta$, that is, $P_{\eta}=[\eta]$. There exists a clear connection between the random polygons $P_n$ and $P_{\eta}$, namely, $P_n$ has the same distribution as $P_\eta$, given that the number of points $N=\#\eta$ of $\eta$ is equal to $n$:
$$
P_n\eqdistr(P_{\eta}\,|\,\#\eta = n).
$$
The Poisson random polygons $P_{\eta}$ have been intensively studied. B{\'a}r{\'a}ny and Reitzner \cite{BR10} and Yukich and Calka \cite{CY17} showed that the number $f_0(P_\eta)$ of vertices of $P_\eta$ satisfies
\begin{align*}
	\EE f_0(P_{\eta})&={2 \ell \over 3} \log n (1+o(1)), 
	\\
	\var f_0(P_{\eta})&= c_2 \ell \log n (1+o(1)) 
\end{align*}
for some constant $c_2>0$.
In fact, more general results have been proven for arbitrary dimensions.
For the first formula see \cite[Theorem 1.2]{BR10}, where also a lower bound for the variance was obtained, and an upper bound for the variance appears in \cite[Theorem 1.1]{BR10a}. For the precise asymptotics for the variance see \cite[Theorem 1.3]{CY17}. In addition, in \cite {BR10a} a Berry-Esseen bound for $f_0(P_\eta)$ is shown, which involves a double logarithmic factor as in \eqref{eq:BaraReitz}. As explained above, we will give more precise estimates for the moments in Theorem \ref{thm:BerryEsseenPoisson} and remove the double-logarithmic factor in the central limit theorem for this Poisson random polygon model $P_\eta$, and eventually carry these results to $P_n$.

\subsection{Step 2: Fixing the construction} 

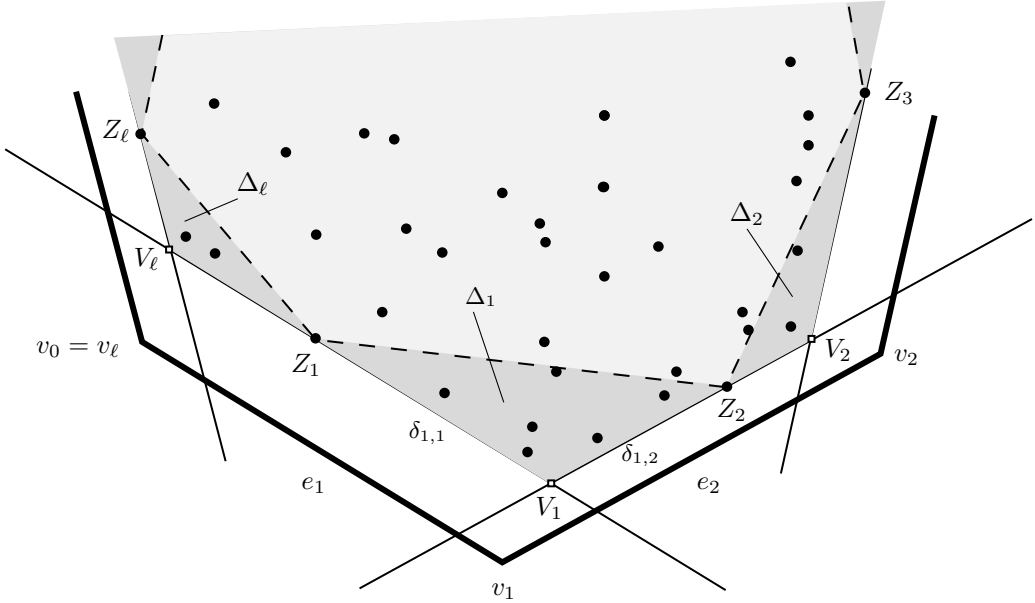
\begin{figure}[t]
\centering
\begin{tikzpicture}
	\node at (0,0) {\includegraphics[width=0.8\textwidth]{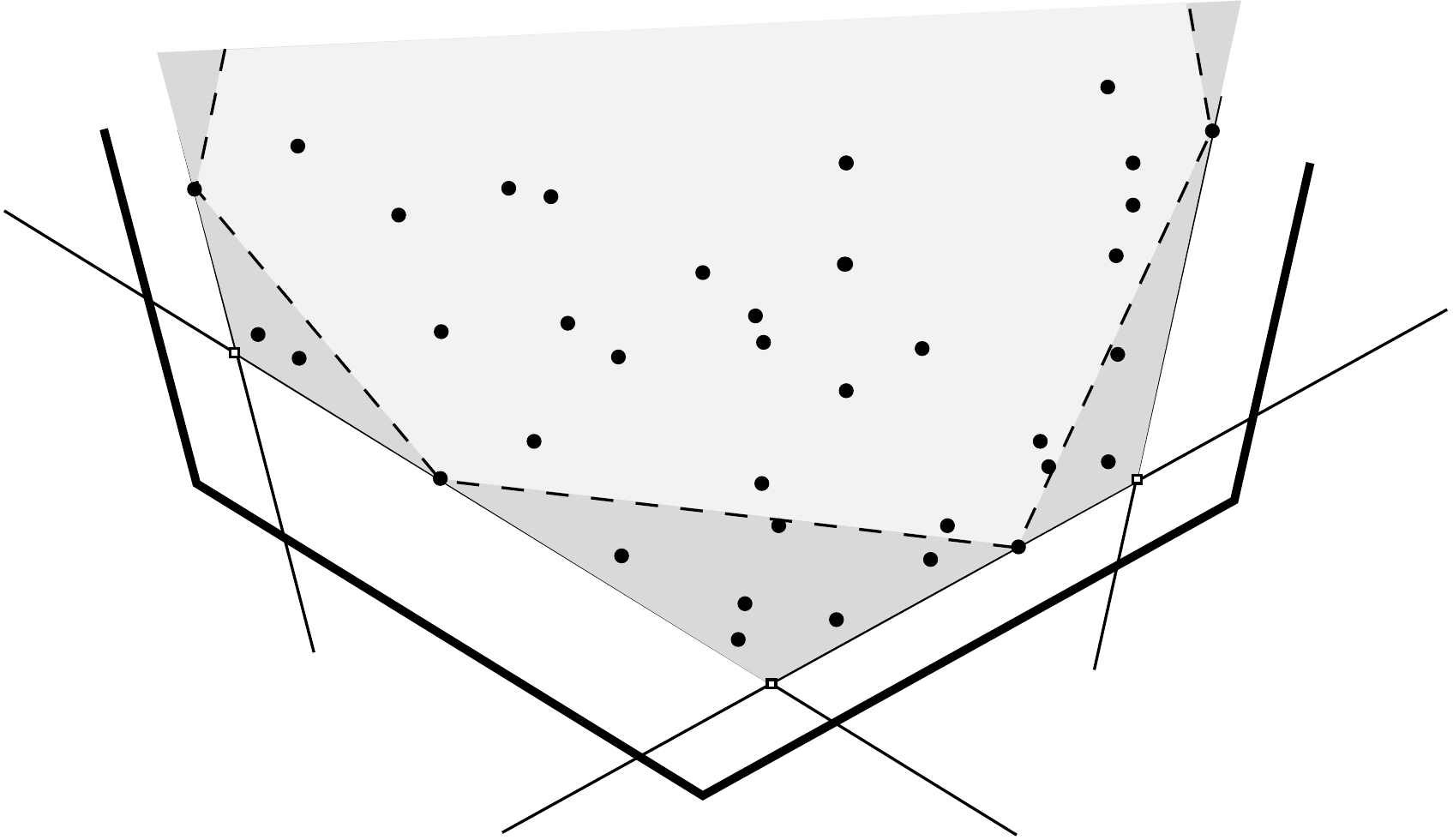}};
	\node at (-2.8,-0.9) {\small $Z_1$};
	\node at (2.8,-1.5) {\small $Z_2$};
	\node at (5,2.7) {\small $Z_3$};
	\node at (-5.3,2.2) {\small $Z_\ell$};
	
	\node at (0.4,-2.8) {\small $V_1$};
	\node at (4.2,-0.7) {\small $V_2$};
	\node at (-4.9,0.5) {\small $V_\ell$};
	
	\node at (-5.8,-0.7) {\small $v_0=v_{\ell}$};
	\node at (-2.7,-2.5) {\small $e_1$};
	\node at (-0.2,-3.9) {\small $v_1$};
	\node at (2.5,-2.5) {\small $e_2$};
	\node at (5.1,-0.8) {\small $v_2$};
	
	\node at (-1.2,-1.8) {\footnotesize $\delta_{1,1}$};
	\node at (1.6,-2.1) {\footnotesize $\delta_{1,2}$};
	
	\node at (-0.5,0) {\footnotesize $\Delta_1$};
	\draw (-0.2,-1.4) -- (-0.6,-0.2);

	\node at (3,1.1) {\small $\Delta_2$};
	\draw (3,0.9) -- (3.6,0);

	\node at (-3.5,1.5) {\small $\Delta_\ell$};
	\draw (-3.7,1.3) -- (-4.4,1);
\end{tikzpicture}
\caption{Illustration of the construction in Step 2.}
\label{fig:Step2}
\end{figure}

Let $v_1,\ldots,v_\ell$ be the $\ell\geq 3$ consecutive vertices of the polygon $P$ and $e_1,\ldots,e_\ell$  be the $\ell$ consecutive edges of $P$, where $e_i =[v_{i-1}, v_i]$ with the convention that here and in what follows the index is taken modulo $\ell$, e.g. $v_0=v_\ell$. Further, denote by $\ell_i$ the length of the edge $e_i$, by $\alpha_i$ the angle of $P$ at the vertex $v_i$ and by $F_i$ the area of the triangle $[v_{i-1}, v_i, v_{i+1}]=[e_i, e_{i+1}]$, $1\leq i\leq \ell$.  For every edge $e_i$, $1\leq i\leq \ell$, with probability one there is one point $Z_i:=(x_i,y_i)$ from the Poisson point process $\eta$, such that either $L^+(Z_i)\cap \eta=\{Z_i\}$ or $L^-(Z_i)\cap \eta=\{Z_i\}$, where $L(Z_i)$ denotes the line parallel to $e_i$ passing through $Z_i$, see Figure \ref{fig:Step2}. Without loss of generality we assume that $L(Z_i)$ is parametrized in a way such that $L^-(Z_i)\cap \eta=\{Z_i\}$ and thus $\eta \subset L^+(Z_i)$. It might happen that $\eta=\varnothing$ or that for some or few $1\leq i\leq \ell $ we have $Z_i=Z_{i+1}$ with $Z_{\ell +1}=Z_1$. Since by a conditioning argument we will exclude these situations in the next steps, we will ignore these cases in the forthcoming discussion.

Denote the point of intersection of $L(Z_i)$ and $L(Z_{i+1})$ by $V_{i}:=L(Z_i)\cap L(Z_{i+1})$, see Figure \ref{fig:Step2}, and put $\Delta_i:=[Z_i,V_{i},Z_{i+1}]$, $1\leq i\leq \ell$, and $\Delta_i := Z_i$ if $Z_i=Z_{i+1}$. 
The length of the edges of the triangle $\Delta_i$ are denoted by 
$\delta_{i,i}=\|Z_i-V_i\| $ and $\delta_{i,i+1}= \| V_i - Z_{i+1} \| $.

In the next sections we will need the first moments and covariances of the logarithmic area of $\Delta_i$, and the probability that the triangle is not too small.
We start with the following lemma.

\begin{lemma}\label{lm:smalldistanceedge}
There is a constant $c_P>0$ depending only on $P$, such that for $1\leq i\leq \ell$ we have 
$$
\PP( d(Z_{i}, e_i) \geq x ) \leq e^{-  c_P n x}.
$$
\end{lemma}

\begin{proof}
Recall that $\ell_i$ denotes the length of $e_i$.
If $d(Z_{i}, e_i) \geq x $ then $L^-(Z_i) \cap P$ contains a triangle with base $\ell_i$ and height $x$, implying that
$$\area( L^-(Z_i) \cap P) \geq  \frac{ \ell_i}2 x  .$$
Because $\eta \subset L^+(Z_i)$, the interior of the set $L^-(Z_i) \cap P$ is empty,  and hence
$$ \PP( d(Z_{i}, e_i) \geq x ) 
\leq e^{- n \, \frac{\ell_i}2 x}.
$$
Thus, the result follows with $c_P:= \min_i \frac{\ell_i}2 $.
\end{proof}

From now on we assume that $d(Z_i, e_i) \leq n^{- \frac 34}$, $Z_i \neq Z_{i+1}$, and $\delta_{i,j} \geq n^{- \frac 14}$ for $1\leq i\leq \ell$ and $j=i, i+1$ (note that the latter conditional already implies that $Z_i\neq Z_{i+1}$, which we have included only for convenience). We denote this event by $\cE$. Observe that given $\cE$ we have 
\begin{equation}\label{eq:Delta>n-12}
\Delta_i \geq \frac 12 \min_{1\leq j\leq \ell} \sin \alpha_j \  n^{- \frac 12} .
\end{equation}
%where $\alpha_j $ is the angle of $P$ between $e_j$ and $e_{j+1}$ at $v_j$.
In a first step we show that $\cE$ happens with high probability. 

\begin{corollary}\label{cor:complcE}
There are constants $\underline c_P, \overline c_P>0$ only depending on $P$, such that
$$
\underline c_Pn^{- 1}\leq \PP(\overline{\cE})\leq \overline c_P n^{-\frac 14}.
$$
\end{corollary}
\begin{proof}
The event that $d(Z_i, e_i) \geq n^{- \frac 34}$ for some $i=1,\ldots,\ell$ occurs according to Lemma \ref{lm:smalldistanceedge}  with probability $\leq \ell e^{- c_P n^{ \frac 14} }$. The probability of the event that $V_i=Z_i=Z_{i+1}$ for some $i=1,\ldots,\ell$, again according to Lemma \ref{lm:smalldistanceedge}, can be estimated by
\begin{align*}
\sum_{i=1}^\ell\PP(Z_i=Z_{i+1}) 
& \leq  
\sum_{i=1}^\ell \PP(Z_i=Z_{i+1}, d(Z_i, e_i) \leq n^{- \frac 34}, d(Z_{i+1}, e_{i+1}) \leq n^{- \frac 34}) + \ell e^{- c_P n^{\frac 14}} 
\\ & \leq  
\sum_{i=1}^\ell \PP(\eta \cap B(v_i,(\sin(\alpha_i/2))^{-1}n^{- \frac 34}) \geq 1) + \ell e^{- c_P n^{\frac 14}}
\\ & \leq
\ell \big[1-\exp\big(-2 \pi\max_{1\leq i\leq \ell}(\sin(\alpha_i/2))^{-1} n^{- \frac 12} \big)\big] + \ell e^{- c_P n^{\frac 14}} 
= O(n^{- \frac 12}).
\end{align*}
Finally, we compute $\PP(\delta_{i,j} \leq n^{- \frac 14})$ for $i=1,\ldots,\ell$ and $j=i,i+1$ using the multivariate Mecke formula \eqref{eq:Mecke}. Denote by $L_i(x)$ the line through $x$ parallel to $e_i$ and by $L_i^+(x)$ the corresponding halfplane not containing the edge $e_i$. 
Taking for simplicity $i=1, j=2$ we have for $\delta_{1,2}=\delta_{1,2}(\eta)$ that 
\begin{align*}
&\PP \big( 
\delta_{1,2} \leq n^{- \frac 14},  \ d(Z_k, e_k) \leq n^{- \frac 34} \ \forall k=1,\ldots,\ell \big) 
\\ & =
\EE \sum_{(X_1, X_2,X_3) \in \eta_{\neq}^2} \ind (\delta_{1,2}(\eta)\leq n^{- \frac 14}) \ind\big(\eta \subset L_1^+(X_1) \cap L_2^+ (X_2)\cap L_3^+ (X_3), d(Z_k, e_k) \leq n^{- \frac 34} \ \forall k=1,\ldots,\ell \big) 
\\ &= 
n^3 \EE \int\limits_{P}\int\limits_{P}\int\limits_{P} \ind(\delta_{1,2}(\eta\cup\{x_1,x_2,x_3\}) \leq n^{- \frac 14}) \\
&\hspace{1.5cm}\times\ind \big(\eta\cup\{x_1,x_2,x_3\} \subset L_1^+(x_1) \cap L_2^+ (x_2)\cap L_3^+ (x_3) , d(Z_k, e_k) \leq n^{- \frac 34} \ \forall k=1,\ldots,\ell \big) \,dx_1dx_2dx_3.
%\\ &= 
%n^2 \EE \int\limits_{P^2} \ind(\delta_{1,2} \leq n^{- \frac 14}) \ind \big(\eta\in L_1^+(x_1) \cap L_2^+ (x_2) , d(Z_i, e_i) \leq n^{- \frac 34} \ \forall i \big) \,dx_1dx_2.
\end{align*}
Let us remark that the condition that $\eta\subset L_1^+(X_1) \cap L_2^+ (X_2)\cap L_3^+ (X_3)$ in the second line automatically implies that $Z_1=X_1$, $Z_2=X_2$ and $Z_3=X_3$. Also, the random variable $\delta_{1,2}(\eta)$ is determined by the random points $X_1$ and $X_2$ in the second and $\delta_{1,2}(\eta\cup\{x_1,x_2,x_3\})$ by the fixed points $x_1$ and $x_2$ in the third line. Now, we first integrate $x_2$ along the segment $L_2(x_2) \cap P$ where everything is fixed except the position of $x_2$ on $[V_1, V_2] \subset L_2(x_2)$, where $V_1$ and $V_2$ are points on $L_2(x_2)$ whose positions are determined by $x_1,x_2$ and $x_3$, see Figure \ref{fig:Step2} (in fact, the dependence of $V_2$ on the position of $x_3$ is the reason why we consider the three points $x_1,x_2,x_3$). Because $ \int_0^a \ind(x \leq n^{- \frac 14}) \, dx \leq n^{- \frac 14}  a^{-1} \int_0^a dx $ and $\| V_1-V_2 \| = \ell_2(1+O(n^{- \frac 34}))$, 
\begin{align*}
&\PP \Big( 
\delta_{1,2} \leq n^{- \frac 14},  \ d(Z_k, e_k) \leq n^{- \frac 34} \ \forall k=1,\ldots,\ell \Big) 
\\ & \leq
n^{- \frac 14} \ell_2^{-1} (1+O(n^{- \frac 34})) n^3  \EE \int\limits_{P}\int\limits_{P}\int\limits_{P}  \ind \big(\eta \in L_1^+(x_1) \cap L_2^+ (x_2)\cap L_3^+ (x_3) , \ldots\\
&\hspace{5cm}\ldots d(Z_k, e_k) \leq n^{- \frac 34} \ \forall k=1,\ldots,\ell \big)\ dx_1dx_2dx_3
\\ & =
n^{- \frac 14} \ell_2^{-1} (1+O(n^{- \frac 34})) \PP\Big( d(Z_k, e_k) \leq n^{- \frac 34} \ \forall k=1,\ldots,\ell \Big),
\end{align*}
where in the last line we have applied backwards the same argument, based on the multivariate Mecke formula, as above. Hence
\begin{align*}
&\PP (\exists i\in\{1,\ldots,\ell\},j\in\{i,i+1\}:\delta_{i,j} \leq n^{- \frac 14})\\
& \leq 
\sum_{i=1}^{\ell}\sum_{j=i,i+1} \PP \Big( \delta_{i,j} \leq n^{- \frac 14}, \ d(Z_k, e_k) \leq n^{- \frac 34} \ \forall k=1,\ldots,\ell \Big) + \ell e^{- c_P n^{\frac 14}}  \\
&= O(n^{- \frac 14}).
\end{align*}
Combining the three estimates yields the right hand side of the inequality.

On the other hand, there is some small $c_1>0$ depending on $P$ such that the probability that $B(v_1, n^{-1}) \cap P$ contains precisely one point of $\eta$ is 
$$ 
\PP(\eta \cap B(v_1, n^{- 1}) = 1) = n\, \area(B(v_1, n^{- 1}) \cap P) e^{- n\, \area(B(v_1, n^{- 1}) \cap P) }
= c_1 n^{- 1} e^{-c_1 n^{- 1}} 
$$
for $n$ sufficiently large. 
The area of the parallel strips along the edges $e_1$ and $e_{2}$ of width $n^{- 1}$ without the disk $B(v_1, n^{- 1})$ is upper bounded by $c_2 n^{-1}$ with some $c_2>0 $ depending on $P$.
Hence, the probability that this region contains no points of $\eta$ is lower bounded by 
$ e^{- c_3} $ with $c_3>0$ depending on $P$. If these independent events occur than the single point in $B(v_1, n^{- 1})$ is just $Z_1=Z_2$, which proves the left hand side inequality, that is
\begin{align*}
\PP(\overline \cE ) 
\geq 
\PP(Z_1=Z_{2}) 
& \geq  
c_1 n^{- 1} e^{-c_1 n^{- 1} -c_3}  
\geq \underline{c}_P n^{- 1}.
\end{align*}
This completes the argument.
\end{proof}

In the second step we investigate the moments and mixed moments of $\log \delta_{i,i}$ and $\log \delta_{i,i+1}$ under the condition $\cE$. 

\begin{lemma}\label{lm:distancesmall}
The logarithmic moments satisfy, for $k=i, i+1$, $1\leq i\leq \ell$,
\begin{align}
\label{eq:E-log-delta}
\EE (\log \delta_{i,k} |\cE ) &=\log \ell_k -1 + O(n^{- \frac 34}) 
\intertext{and}	
\label{eq:E-log-delta-sqr} \EE ((\log \delta_{i,k} )^2 |\cE) &=(\log \ell_k-1)^2  + 1 +O(n^{- \frac 34}) .
\intertext{For the mixed logarithmic moments we obtain}
\label{eq:E-log-delta-neighbors}
\EE[ (\log \delta_{i,i+1}) (\log \delta_{i+1,i+1})|\cE ] &=  
(\log \ell_{i+1} -1)^2  + 1- \frac {\pi^2}6 + O(n^{- \frac 34})
\intertext{and for $j \neq l$,}
\label{eq:E-log-delta-nonneighbors}
\EE[ (\log \delta_{i,j}) (\log \delta_{k,l})|\cE ]& =  
(\log \ell_{j} -1)(\log \ell_{l} -1) + O(n^{- \frac 34}). 
\end{align}
\end{lemma}
\begin{proof}
We assume the event $\cE$ and prove \eqref{eq:E-log-delta} e.g.\ for $\delta_{i,i+1}$ and $i=1$.  Our argument will be similar to that in the proof of Corollary \ref{cor:complcE} and we will also use the same notation introduced as there. In particular, recall that $L_i(x)$ is the line through $x$ parallel to $e_i$ and $L_i^+(x)$ the corresponding halfplane not containing the edge $e_i$. 
The multivariate Mecke formula \eqref{eq:Mecke} yields
\begin{align*}
\EE (\log \delta_{1,2} \ind(\cE)) 
& =
\EE \sum_{(X_1, X_2, X_3) \in \eta_{\neq}^3} \log \delta_{1,2}(\eta)\, \ind(\eta \subset L_1^+(X_1) \cap L_2^+ (X_2) \cap L_3^+ (X_3) , \cE) 
\\ &= 
n^3 \EE \int\limits_{P}\int\limits_{P}\int\limits_{P} \log \delta_{1,2}(\eta\cup\{x_1,x_2,x_3\})  \\
&\hspace{3cm}\times\ind(\eta\cup\{x_1,x_2,x_3\}\subset L_1^+(x_1) \cap L_2^+ (x_2) \cap L_3^+ (x_3) , \cE)\, d x_2 d x_3 d x_1 .
\end{align*}
We integrate $x_2$ on the line segment $[V_1,V_2]\subset L_2(x_2) \cap P$, where $V_1$ and $V_2$ are the points determined by $x_1,x_2,x_3$ as in Figure \ref{fig:Step2}. Because $ \int_0^a \log x\, dx = (\log a - 1) a $ and $\| V_1-V_2 \| = \ell_2(1+O(n^{- \frac 34}))$, given $\cE$, we have that
\begin{eqnarray*}
\int\limits_{L(x_2) \cap P} \log \delta_{1,2}(\eta\cup\{x_1,x_2,x_3\}) \, d_{L_2(x_2)}x_2 
&=&
\int\limits_0^{\|V_1-V_2\|} \log x \, dx \\
&=&
(\log \ell_2 - 1 +O(n^{- \frac 34})) \ell_2 
\\ &=&
(\log \ell_2 - 1 +O(n^{- \frac 34})) \int\limits_{L(x_2) \cap P} d_{L_2(x_2)}x_2 
\end{eqnarray*}
under the event $\cE$, where $d_{L_2(x_2)}x_2 $ indicates that we are integrating with respect to the Lebesgue measure on $L_2(x_2)$. 
Thus, 
$$
\EE ( \log \delta_{1,2} \ind(\cE)) 
 =
(\log \ell_2 - 1 +O(n^{- \frac 34})) \EE \ind (\cE),
$$
and 
$$
\EE ( \log \delta_{1,2} |\cE) 
= 
\log \ell_2 - 1 +O(n^{- \frac 34}).
$$
In this way we obtain Equation \eqref{eq:E-log-delta}.
In the same way, this time using  the identity
$ \int_0^a (\log x)^2 dx = ((\log a)^2 - 2 \log a +2) a $,
one proves \eqref{eq:E-log-delta-sqr}.

Analogously, using the identity
\begin{eqnarray*}
\int\limits_0^a \log x \log(a-x)  dx
&=& 
a \left[(\log a )^2 - 2 \log a   + 2- \frac {\pi^2}6 \right]
\end{eqnarray*}
we obtain for the expectation of the product of the two log-neighbouring distances 
$$
\EE ((\log \delta_{1,2})( \log \delta_{2,2} )| \cE ) 
=
(\log \ell_2 )^2 - 2 \log \ell_2   + 2- \frac {\pi^2}6 + O(n^{- \frac 34})
=
(\log \ell_2 -1)^2  + 1- \frac {\pi^2}6 + O(n^{- \frac 34}). 
$$
Considering the product of two distances not on the same line $L_i(\,\cdot\,)$, e.g.
$$
\EE (\log \delta_{1,1} \log \delta_{1,2} | \cE ) 
\qquad\mbox{ or }\qquad
\EE (\log \delta_{1,2} \log \delta_{2,3} | \cE ) 
$$
we rewrite this as a multiple integral using the multivariate Mecke formula once again, and integrate first with respect to $x_2$ on $L_2(x_2)$ to obtain
\begin{eqnarray*}
\EE ((\log \delta_{1,1})( \log \delta_{1,2} )| \cE ) 
&=&
(\log \ell_2 - 1 +O(n^{- \frac 34}))  \EE (\log \delta_{1,1} | \cE ) 
\\ &=&
(\log \ell_2 - 1 )  (\log \ell_1 - 1 ) +O(n^{- \frac 34}), 
\end{eqnarray*}
and similarly for all other cases. This shows that $\delta_{i,i+1}$ and $\delta_{i+1,i+1}$ are asymptotically  uncorrelated to all other $\delta_{j,k}$ not on $L_{i+1}(\,\cdot\,)$. The proof is thus completed.
\end{proof}
The conditional second moments can be written in a more concise way as the conditional covariance of the involved quantities.
\begin{corollary}
	For $i,j,k,l\in\{1,\ldots,\ell\}$ it holds that
\begin{equation}
	\label{eq:cov-delta_i}
	\cov (\log \delta_{i,j} , \log \delta_{k,l}|\cE )  =  
	\ind(j=l) - \ind(j=\ell)\ind(|i-k|=1)\frac{\pi^2}6 + O(n^{- \frac 34}). 
\end{equation}
\end{corollary}

For $i\in\{1,\ldots,\ell\}$, the logarithmic area of the triangle $\Delta_i$ equals
$$
\log \area(\Delta_i)
=
 \log \frac{\sin \alpha_i}2 + \log \delta_{i,i}  + \log \delta_{i,i+1} .
$$
Because the edges of $\Delta_i$ of length $\delta_{i,i}$ and $\delta_{i,i+1}$ are parallel to $e_i$ and $e_{i+1}$, we see that 
$$
F_i= \frac{\sin \alpha_i}2 \ell_i \ell_{i+1},\qquad 1\leq i\leq \ell,
$$
is the area of the triangle $[v_{i-1}, v_i, v_{i+1}]=[e_i, e_{i+1}]$.
Lemma \ref{lm:distancesmall} yields immediately the conditional expectation, variance and covariances of $\log\area(\Delta_i)$. For example, 
$$
\cov (\log \area(\Delta_i), \log \area (\Delta_k)|\cE)
= 
\sum_{j=i, i+1,\ l=k, k+1} \cov (\log \delta_{i,j}  , \log \delta_{k,l}  |\cE).
$$
Combined with the formula for the conditional covariance \eqref{eq:cov-delta_i} this yields the following result.

\begin{corollary}\label{cor:moments-logarea}
For $i\in\{1,\ldots,\ell\}$ one has that
\begin{align}
\label{eq:E-log-area}\EE (\log \area (\Delta_i )|\cE ) &= \log F_i -2  + O(n^{- \frac 34}) , \\
\label{eq:Var-log-area} \var (\log \area (\Delta_i) |\cE ) &= 2+  O(n^{- \frac 34}) 
\intertext{and}
\label{eq:Covar-log-area} \cov (\log \area( \Delta_i), \log \area (\Delta_k) |\cE ) &= \ind(|i-k|=1) \left(1 - \frac{\pi^2}6 \right) +  O(n^{- \frac 34}) .
\end{align}
\end{corollary}

\subsection{Step 3: Berry-Esseen bound for the Poisson random convex chain}

Let $T$ be the canonical triangle with vertices $(0,1)$, $(0,0)$ and $(0,1)$, and $\chi$ a homogeneous Poisson point process with $\EE(\#\chi)=M>1$ in $T$. The convex hull of the two vertices $(0,1), (1,0)$ and the points of $\chi$ is denoted by 
$$
T_{\chi}:= [\chi, (0,1),(1,0)] .$$ 
Denote by $f_0(T_\chi)$ the number of vertices of $T_\chi$.
In order to obtain the Berry-Esseen bound for the Poissonized version of the random convex chain as in \eqref{eq:CLTChain} we will use the following transfer lemma taken from \cite[Lemma 3.2]{BR10}. The proof of this lemma can be found, for example, in \cite{BV06} (see also the remark after Lemma 3.2 in \cite{BR10}).

\begin{lemma}\label{lm:transference}
	Given two sequences of random variables $\xi_n$ and $\xi'_n$ with means $\mu_n\in\RR$ and $\mu_n'\in\RR$, and variances $0<\sigma_n^2<\infty$ and $0<\sigma_n^{'2}<\infty$, respectively. Assume that there are sequences $\varepsilon_1(n)$, $\varepsilon_2(n)$, $\varepsilon_3(n)$ and $\varepsilon_4(n)$, all tending to zero as $n\to \infty$, such that
	\begin{itemize}
		\item[(i)] $|\mu_n'-\mu_n|\leq \varepsilon_1(n)\sigma_n$,
		\item[(ii)] $|\sigma_n^{'2}-\sigma_n^2|\leq \varepsilon_2(n)\sigma_n^{2}$,
		\item[(iii)] for every $x\in\RR$, $|\PP(\xi'\leq x)-\PP(\xi\leq x)|\leq \varepsilon_3(n)$,
		\item[(iv)] for any $x\in\RR$,
		$
		\Big|\PP\Big({\xi_n'-\mu_n'\over \sigma_n'}\leq x\Big)-\Phi(x)\Big|\leq \varepsilon_4(n).
		$
	\end{itemize}
	Then there is a positive constant $C>0$ such that 
	$$
	\sup_{x\in\RR}\Big|\PP\Big({\xi_n-\mu_n\over \sigma_n}\leq x\Big)-\Phi(x)\Big|\leq C\sum_{i=1}^4\varepsilon_i(n).
	$$
\end{lemma}   

 In order to verify conditions (i) and (ii) of the previous lemma for our model we derive the following asymptotic formulas for the expectation and the variance of $f_0(T_{\chi})$. For later purposes we also include the asymptotics for the second moment as well and we formulate our result for general homogeneous Poisson point processes.

\begin{lemma}\label{lm:estimatesPoissonChain}
Consider a homogeneous Poisson point process $\chi$ in the canonical triangle $T$ with $\EE(\# \chi)=M>1$. 
Then
\begin{align}
\EE f_0(T_{\chi})&={2\over 3}\log M +  \frac {2\gamma +7}3 +O(M^{-1/2}),\label{eq:poissonExp}\\
%\EE f_0(T_{\chi})^2&={4\over 9}(\log M)^2 + O(\log M),\label{eq:poisson2moment}\\
\var f_0(T_{\chi})&=  \frac {10}{27} \log M +  \frac{10 \gamma + 2\pi^2 - 28 }{27}  + O(M^{-1/2}) ,\label{eq:poissonVar}
\end{align}
as $M\to\infty$.
\end{lemma}
\begin{proof}
First of all note, that without loss of generality we may assume $M\in\mathbb{N}$, since $\log M= \log\lfloor M\rfloor + (M^{-1})$. Denote by $H_n = \sum_{i=1}^{n} \frac 1i$ the harmonic sum and by $H^{(2)}_n = \sum_{i=1}^{n} \frac 1{i^2}$ the harmonic sum of second order. Set $H_0= H^{(2)}_0:=0$ for convenience. It is well known that 
\begin{align}
H_n &= \log n +\gamma + O(n^{-1}),\label{eq:HarmNum1},\\ 
H^{(2)}_n &= \frac{\pi^2}6+O(n^{-1})\label{eq:HarmNum2},
\end{align}
as $n \to \infty$, where $\gamma$ is the Euler–Mascheroni constant.
Let $T_k$, $k\ge 1$ be the random convex chain, which is build based on the sample of $k$ independent random points $Y_1,\ldots, Y_k$, uniformly distributed inside the canonical triangle $T$. That is, $T_k:=[Y_1,\ldots, Y_k,(0,1),(1,0)]$. It is known from \cite[Corollary 1 and Corollary 2]{Buch12} that
\begin{align}
\label{eq:expUnif}
\EE f_0(T_k)&= {2 \over 3} H_k + \frac 73, 
\\
\label{eq:varUnif}
\var f_0(T_k)&={10\over 27} H_k + {4\over 9} H^{(2)}_k  - {28 \over 27 } +  {4 \over 9(k+1)} ;
\end{align}
note that the result in \cite{Buch12} is stated for the quantity $N_k= f_0(T_k)-2$.
Let $Y$ be a Poisson random variable with mean $M$. Then 
$$
\EE H_Y 
= 
\sum_{k=0}^\infty (H_k-H_M) \PP(Y=k) + H_M .
$$
Because of the trivial estimate 
$$
|H_Y-H_M| \leq \max \left\{ \frac {Y-M}{M}, \frac {M-Y}{Y+1} \right\} 
\leq 
\frac {|Y-M|}M + \frac {|M-Y|}{Y+1}  
$$
we have
\begin{align*}
|\EE H_Y - H_M |
\leq 
\EE \, \frac {|Y-M|}M + \, \EE \frac {|M-Y|}{Y+1}
\leq
\sqrt{\EE \frac {(Y-M)^2}{M^2}} + \sqrt{\EE \frac {(M-Y)^2}{(Y+1)^2}} .
\end{align*}
Since $\EE Y=M$ and  $\EE Y^2 = M^2 + M $,
we see that 
$$
\EE \, \frac{(Y-M)^2}{M^2} = M^{-2}\EE Y^2 - 2 M^{-1} \EE Y +1 = M^{-1}  .
$$
Analogously, since
$
\EE (Y+1)^{-1} 
= M^{-1} \left(1 - e^{-M} \right)
\geq M^{-1} (1-  M^{-1}) $, and 
$$
\EE (Y+1)^{-2} 
= \sum_{k=2}^\infty \frac {M^{k-2}} {k!} \frac {k} {k-1}  e^{-M}
\leq 
\sum_{k=2}^\infty \frac {M^{k-2}} {k!} \left(1+\frac {3} {k+1} \right) e^{-M}
\leq 
M^{-2} (1+ 3 M^{-1})
,
$$
we see that 
\begin{align*}
\EE \frac {(M-Y)^2}{(Y+1)^2}
&=
\EE \frac {(M+1)^2}{(Y+1)^2} -2 \EE \frac {(M+1)}{(Y+1)} + 1 
\leq 17 M^{-1}  .
\end{align*}
Hence, by \eqref{eq:HarmNum1}
$$
\EE H_Y 
= 
H_M + O(M^{-\frac 12}) = \log M + \gamma + O(M^{-\frac 12}).
$$
Together with \eqref{eq:expUnif} this proves \eqref{eq:poissonExp}.
Similarly, by the law of total variance we have
\begin{align*}
\var f_0(T_{\chi})
&=
\EE_Y \var (f_0(T_{k})|Y=k) + \var_Y \EE (f_0(T_k)|Y=k).
\\ &=
\EE \left(\frac {10}{27} H_Y  + {4\over 9} H^{(2)}_Y  - {28 \over 27 } +  {4 \over 9(Y+1)}\right) + \var \left( \frac 23 H_Y \right),
\end{align*}
and by \eqref{eq:HarmNum2} we obtain
\begin{align*}
\var f_0(T_{\chi}) 
&=
\left(\frac {10}{27} \log M +  \frac{10 \gamma + 2\pi^2 - 28 }{27}  + O(M^{-\frac 12}) \right) + \frac 49 \var H_Y .
\end{align*}
To prove \eqref{eq:poissonVar} it remains to show that the variance of $H_Y$ is bounded by a constant. For this we use the Poincar\'e inequality for Poisson random variables, which says that
$$
\var f(Y) \leq M \EE (f(Y+1)-f(Y))^2
$$
for functions $f:\{0,1,2,\ldots\}\to\RR$ for which $\var f(Y)<\infty$. This inequality can be considered as a special case of the general Poincar\'e inequality \cite[Theorem 18.7]{LP} for functionals of Poisson random measures. Applying this to $f(Y)=H_Y$ we conclude that
$$
\var H_Y 
\leq 
M \EE \frac 1{(Y+1)^2}
\leq 
\sum 2 \frac{M^{k+1}}{(k+2)!} e^{-M} \leq 2 M^{-1},
$$
which completes the argument.
\end{proof}

Now we are prepared to prove a Berry-Esseen bound for the number of vertices of the Poisson random chain $T_\chi$ in the canonical triangle $T$. 

\begin{lemma}\label{lm:BerryEssenPoissonChain}
Consider a homogeneous Poisson point process $\chi$ in the canonical triangle $T$ with $\EE(\# \chi)=M>1$. Then
$$
\sup_{x\in\RR}\Big|\PP\Big({f_0(T_\chi)-\EE f_0(T_\chi)\over \sqrt{\var f_0(T_\chi)}}\leq x\Big)-\Phi(x)\Big|\leq {c\over\sqrt{\log M}},
$$
for some absolute constant $c>0$.
\end{lemma}

\begin{proof}
As before, without loss of generality we assume $M\in\mathbb{N}$. Let $T_M$ denotes the convex chain, build on random points $X_1,\ldots, X_M$ independently and uniformly distributed in $T$. We apply Lemma \ref{lm:transference} with $\xi_M':=f_0(T_M)$ and $\xi_M:=f_0(T_{\chi})$. The condition (iv) with $\varepsilon_4(M)=c_4/\sqrt{\log M}$ for some constant $c_4>0$ independent of $M$ follows immediately from \eqref{eq:CLTChain}. The condition (i) with $\varepsilon_1(M)=c_1/\sqrt{\log M}$ for $c_1>0$ independent of $M$ can be verified by formulas \eqref{eq:poissonExp}, \eqref{eq:expUnif} and \eqref{eq:varUnif}. Analogously, condition (ii) with $\varepsilon_2(M)=c_2/\log M$ for $c_2>0$ independent of $M$ follows from \eqref{eq:poissonVar} and \eqref{eq:varUnif}.

For the verification of condition (iii) we use the convex floating body introduced in Section \ref{sec:Prelim}, 
and follow an approach already used somewhat implicitly in \cite{Gro88} as well as \cite{Reitz05}. 
Recall that $\cA_M$ is the event that the vertices of the convex hull of the $M$ random points are contained in the wet part $T(v< b_0 M^{-1}\log M)$ of the triangle $T$, which clearly implies that all vertices of the convex chain $T_M$ are contained in this wet part. 
Analogously, $A^\pi_M$ is the event that all vertices of the Poisson convex hull, and thus all vertices of the Poisson convex chain $T_{\chi}$ belong to $T(v< b_0 M^{-1}\log M)$. 
By Lemma \ref{le:floating-Pn} and Lemma \ref{le:floating-Peta} we have 
$$
\PP(\overline{\cA_M} )  = O(M^{-6})\qquad\text{and}\qquad
\PP(\overline{\cA^\pi_M} ) = O(M^{-6}).
$$

\begin{figure}[t]
	\centering
	\begin{tikzpicture}
		\clip (-4.5,-4.5) rectangle (4.5,4.5);
		\node at (0,0) {\includegraphics[width=0.44\textwidth]{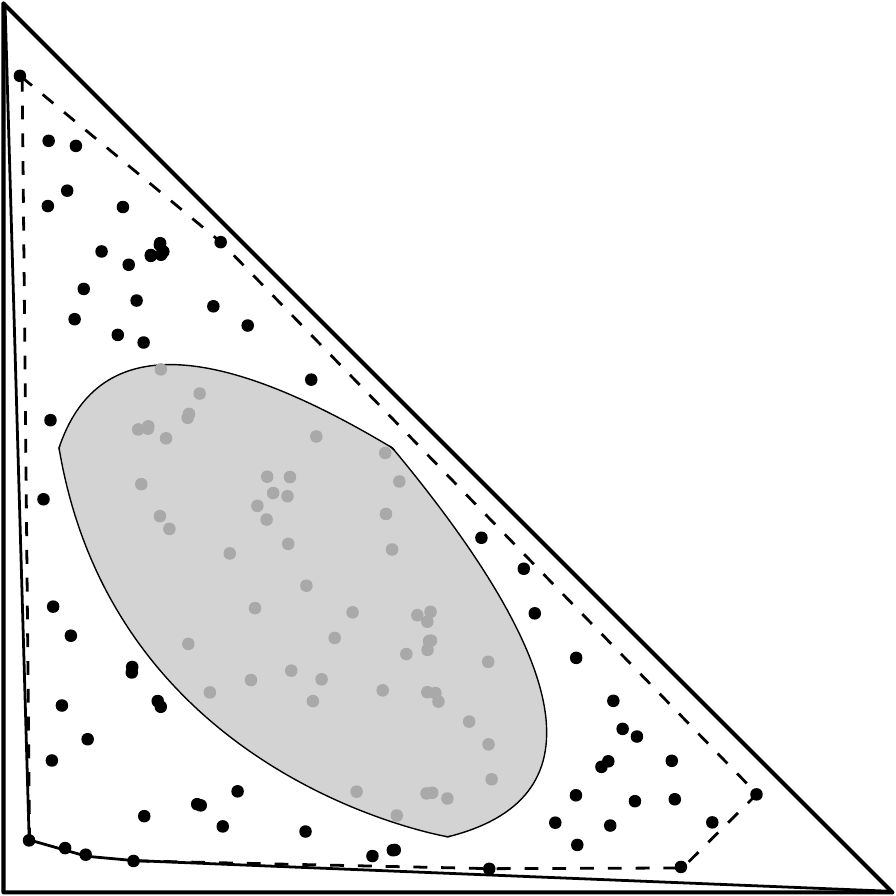}};
		\node at (-1.3,-1) {\tiny $T(v\ge b_0 M^{-1}\log M)$};
		\node at (1.2,0) {\small $P_M$};
		\node at (-4.2,-2.2) {\small $T_M$};
		
		\draw (0.75,0) -- (0.18,-0.43);
		\draw (-4,-2) -- (-3.57,-0.7);
	\end{tikzpicture}
	\caption{Illustration of the construction used in the proof of Lemma \ref{lm:BerryEssenPoissonChain}. The convex hull $P_M$ (or $P_\chi$) is indicated by the dashed segments, while the convex chain $T_M$ (or $T_\chi$) by a solid line. The floating body of $T$ is drawn in grey.}
	\label{fig:Step3}
\end{figure}
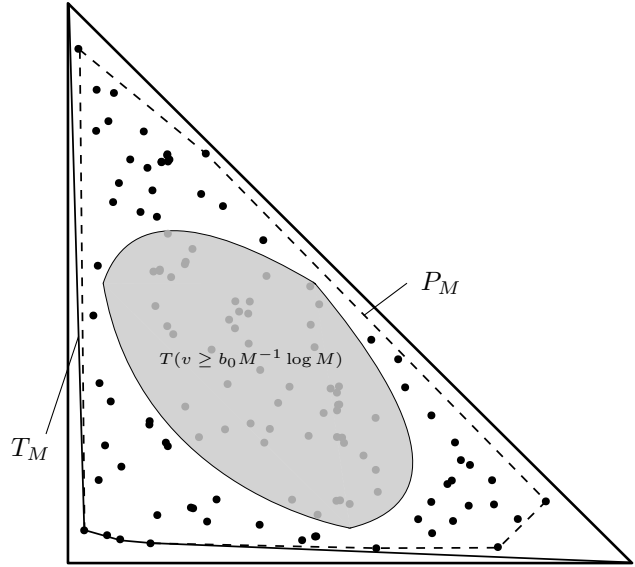

Combining these estimates yields, for any $x\in\RR$,
\begin{equation}\label{eq:24.01.22}
\big|\PP(f_0(T_M)\leq x)-\PP(f_0(T_{\chi})\leq x)\big|
\leq 
\big|\PP(f_0(T_M)\leq x, \cA_M)-\PP(f_0(T_{\chi})\leq x , \cA^\pi_{M})|+ O(M^{-6}) .
\end{equation}
For $T_\chi$ we have that $\#\big(\chi \cap T(v\ge b_0 M^{-1} \log M)\big)$ is Poisson distributed with mean
$$ 
Mp 
:= M\, {\area(T(v\ge b_0 M^{-1} \log M)) \over \area(T)}
= O( (\log M)^2 )
$$
by \eqref{eq:wetpart},  and for $T_M$ the number of points in $T(v\ge b_0 M^{-1} \log M)$ is binomial distributed with mean $p$. Denote by $E_m$ the event that precisely $m$ points of the Poisson or binomial process are in $T(v\ge b_0 M^{-1} \log M)$.
Coupling both processes in the canonical way and using Lemma \ref{le:diff-Poisson-binom} together with \eqref{eq:24.01.22} yields
\begin{align*}
\big|\PP(f_0(T_M) & \leq x)- \PP(f_0(T_{\chi})\leq x)\big|
\\ & \leq 
\sum_{m=0}^\infty \PP(f_0(T_M)\leq x, \cA_M | E_m )\left| \frac{(Mp)^m}{m!} e^{- Mp} - {M \choose m} p^m (1-p)^{n-m} \right| 
+ O(M^{-6})
\\ & \leq
2p + O(M^{-6})\\
& = O( 
M^{-1} (\log M)^2).
\end{align*}
Thus condition (iii) in Lemma \ref{lm:transference} holds with 
$\varepsilon_3(M)=c_3 M^{-1} (\log M)^2$ for some $c_3>0$ independent of $M$. An application of Lemma \ref{lm:transference} finishes the proof.
\end{proof}

\subsection{Step 4: Berry-Esseen bound for the Poisson model under condition $\cE$}

In the next step we consider the random variable $f_0(P_{\eta})$, conditioned on the event $\cE$ we introduced in Step 2:
$$
\xi:=(f_0(P_{\eta})|\cE).
$$
Let us also condition on the positions of the points $Z_1,\dots,Z_\ell$ and introduce the random variable
$$
\xi':=(f_0(P_{\eta})|\cE, Z_1,\ldots,Z_\ell).
$$ 

It should be mentioned that under $\cE$ all $\ell$ points $Z_1,\ldots, Z_\ell$ are well defined and in fact distinct. Hence, due to the independence property of Poisson point processes the random variable $\xi'$ can be decomposed into the sum of $\ell$ independent random variables $\xi'_i$, $1\leq i\leq \ell$, where each $\xi_i'$ is defined as a number of vertices of the random convex chain, formed by the Poisson point process $\eta$ restricted to the triangle $\Delta_i$. More precisely, given an arbitrary triangle $\Delta\subset P$ with vertices $v_1,v_2,v_3$ and a Poisson point process $\eta$ we define
$$
T_{\eta}(\Delta,v_1,v_2):=[(\eta\cap\Delta), v_1,v_2].
$$
Then we take
$$
\xi'_i:= f_0(T_{\eta}(\Delta_i,Z_i,Z_{i+1}))-2,
$$
where the $-2$ is coming from the fact that we exclude the two endpoints $Z_i$ and $Z_{i+1}$ of the convex chain. Now, consider for each $1\leq i\leq \ell$ the affine transformation $A_i:\RR^2\to\RR^2$ which maps the triangle $\Delta_i$ with vertices $Z_i$, $V_i$ and $Z_{i+1}$ to the canonical triangle $T$ with vertices $(0,1)$, $(0,0)$ and $(0,1)$. Using the mapping property \cite[Theorem 5.1]{LP} and restriction property \cite[Theorem 5.2]{LP} of Poisson point processes we conclude that $\eta_i:=A_i(\eta\cap\Delta_i)$ is a homogeneous Poisson point process on $T$ with intensity $2n\,\area(\Delta_i)$.

Since the number of vertices is invariant under affine transformations we conclude that
$$
\xi'_i\eqdistr f_0(T_{\eta_i})-2.
$$

In order to prove a Berry-Esseen bound for the random variable 
$$ \xi'=\sum_{i=1}^\ell \xi_i'+ \ell,   $$
where the additional summand $+\ell$ is coming from the fact that we excluded in the definition of $\xi_1,\ldots,\xi_\ell$ the points $Z_1,\ldots,Z_\ell$, we will use the following lemma.

\begin{lemma}\label{lm:GlueBerryEsseen}
Let $X_1,\ldots,X_k$ be independent random variables with $\mu_i:=\EE X_i<\infty$, $\sigma_i:=\sqrt{\var X_i}\in(0,\infty)$, $1\leq i\leq k$ and let $G_1,\ldots,G_k$ be independent standard Gaussian random variables. Let $\varepsilon_i>0$, $1\leq i\leq k$ be such that
\begin{equation}\label{eq:conditions}
\sup_{x\in\RR}\Big|\PP\Big({X_i-\mu_i\over \sigma_i}\leq x\Big)-\PP(G_i\leq x)\Big|\leq \varepsilon_i,\qquad 1\leq i\leq k.
\end{equation}
Then for $X:=X_1+\ldots+X_k$ we have
$$
\sup_{x\in\RR}\Big|\PP\Big({X-\EE X\over \sqrt{\var X}}\leq x\Big)-\Phi(x)\Big|\leq \sum_{i=1}^k\varepsilon_i.
$$
\end{lemma}

\begin{proof}
Let $\cF$ be the space of cumulative distribution functions, namely
$$
\cF:=\{F:\RR\mapsto[0,1]\colon F\text{ is right-continuous, monotonically increasing}, F(\infty)=1,F(-\infty)=0\},
$$
where $F(\pm\infty)$ has to be interpreted as the appropriate limit. First of all let us recall that the classical Kolmogorov (or uniform) metric $d:\cF\times\cF\to [0,\infty)$ on the space $\cF$ is defined as
$$
d(F_1,F_2):=\sup_{x\in\RR}|F_1(x)-F_2(x)|.
$$
Given two random variables $X,Y$ with cumulative distribution functions $F_X,F_Y$, respectively, we write
$$
d(X,Y):=d(F_X,F_Y)=\sup_{x\in\RR}|\PP(X\leq x)-\PP(Y\leq x)|.
$$
Using this notation and the fact that $d(aX+b,aY+b)=d(X,Y)$ for any $a>0,b\in\RR$ the conditions in \eqref{eq:conditions} can be written in the form
\begin{equation}\label{eq:conditionsNew}
d(X_i-\mu_i,\sigma_i G_i)\leq \varepsilon_i,\qquad 1\leq i\leq k.
\end{equation}
Further, we apply the so-called semi-additivity property of the Kolmogorov metric, which says that for any independent $Y_1,\ldots,Y_k$ and any independent $Y'_1,\ldots,Y'_k$ one has that
$$
d(Y_1+\ldots+Y_k,Y_1'+\ldots+Y_k')\leq \sum_{i=1}^n d(Y_i,Y_i'),
$$
see \cite[Section 2.3, Equation (1.2)]{Zol76}. Recalling that $X=X_1+\ldots+X_k$, and taking $Y_i=X_i-\mu_i$ and $Y_i'=\sigma_i G_i$ we conclude by \eqref{eq:conditionsNew} that
\begin{align*}
d(X-\EE X, \sigma_1 G_1+\ldots+\sigma_k G_k)&=d\Big({X-\EE X\over \sqrt{\var X}}, {\sigma_1 G_1+\ldots+\sigma_k G_k\over (\sigma_1^2+\ldots+\sigma_k^2)^{1/2}}\Big)\leq \sum_{i=1}^k\varepsilon_i.
\end{align*}
Finally, we need to observe that $(\sigma_1 G_1+\ldots+\sigma_k G_k)/(\sigma_1^2+\ldots+\sigma_k^2)^{1/2}$ has the standard Gaussian distribution. This completes the proof.
\end{proof}

We apply Lemma \ref{lm:GlueBerryEsseen} with $k=\ell$ to the random variables $X_1:=\xi_1',\ldots,X_\ell:=\xi_\ell'$. 

\begin{corollary}\label{cor:PoissonWithZ}
There exists a constant $c>0$ such that for any $n\geq 2$ we have that
$$
\sup_{x\in\RR}\Big|\PP\Big({\xi'-\EE\xi'\over \sqrt{\var \xi'}}\leq x\Big)-\Phi(x)\Big|\leq  {c\over \sqrt{\log n}}.
$$
\end{corollary}
\begin{proof}
Lemma \ref{lm:BerryEssenPoissonChain} yields \eqref{eq:conditions} for the random variables $X_i=\xi'_i$ with $\varepsilon_i=C(\log n +\log\area(\Delta_i))^{-1/2}$, $1\leq i\leq \ell$ and some absolute constant $C>0$. It is clear that for $X := \xi'-\ell = \sum_{i=1}^\ell \xi'_i$ we have $X-\EE X=\xi'-\EE\xi'$ and $\var X=\var\xi'$. Moreover, according to \eqref{eq:Delta>n-12} we have 
$\Delta_i \geq c n^{- \frac 12}$ given $\cE$, and thus 
$\varepsilon_i = O((\log n)^{-1/2})$ for all $1\leq i\leq \ell$, and the proof is complete. 
\end{proof}

Note that the obtained bound is independent of the exact position of the points $Z_1,\ldots, Z_\ell$ if we condition on the event $\cE$. This already indicates that the same bound holds for the random variable $f_0(P_\eta)$ conditionally on $\cE$ only. Our next result ensures that this is indeed the case.

\begin{lemma}\label{lem:RemovePointsZ}
There exists a constant $c>0$ such that for any $n\geq 2$ we have
$$
\sup_{x\in\RR}\Big|\PP\Big({(f_0(P_{\eta})|\cE)-\mu \over\sigma}\leq x\Big)-\Phi(x)\Big|\leq {c\over\sqrt{\log n}}
$$
with 
\begin{align*}
\mu &=\EE (f_0(P_{\eta})|\cE) =
{2 \ell \over 3}\log n + \frac 23 \sum_{i=1}^\ell \log F_i  + \frac{2 \gamma \ell }3  + O(n^{- \frac 12}) 
\intertext{and}
\sigma^2 &=\var (f_0(P_{\eta})|\cE) 
= {10 \ell \over 27}\log n +{10\over 27}\sum_{i=1}^\ell \log F_i + 
\frac{(10 \gamma -2 \pi^2 )\ell }{27} 
+ O(n^{-\frac 12}).
\end{align*}
\end{lemma}
\begin{proof}
Let ${\bf Z}=(Z_1, \dots, Z_\ell)$ and recall that $\xi=(f_0(P_{\eta})|\cE)$ and $\xi'=(f_0(P_{\eta})|\cE, {\bf Z})$. Moreover, we define $\mu:=\EE\xi$, $\mu':=\EE\xi'$ and $\sigma^2:=\var \xi$, ${\sigma'}^2:=\var \xi' $
(in this proof we suppress the dependence on the parameter $n$ in our notation). Using the representation $\xi'=\sum_{i=1}^\ell \xi_i'+\ell $ together with Lemma \ref{lm:estimatesPoissonChain}, thanks to the condition $\cE$, we obtain
\begin{align*}
\mu'
&=
{2 \ell \over 3}\log n +{2\over 3}\sum_{i=1}^\ell \log\area(\Delta_i)+\frac{(2 \gamma + 4)\ell }3  + O(n^{-\frac 12}),\qquad 
\\
\sigma'^2
&=
{10 \ell \over 27}\log n +{10\over 27}\sum_{i=1}^\ell \log\area(\Delta_i) + 
\frac{(10 \gamma + 2\pi^2 - 28)\ell }{27}  + O(n^{-\frac 12}) .
\end{align*}
Corollary \ref{cor:moments-logarea} shows that
\begin{align*}
\mu 
= 
\EE (\mu' | \cE)
 &= 
\EE \left( {2 \ell \over 3}\log n +{2\over 3}\sum_{i=1}^\ell \log\area(\Delta_i)+\frac{(2 \gamma + 4)\ell }3  + O(n^{-\frac 12}) \ \Big|\cE \right)
\\ &=
{2 \ell \over 3}\log n + \frac 23 \sum_{i=1}^\ell \log F_i  + \frac{2 \gamma \ell }3  + O(n^{- \frac 12}),
\end{align*}
which already coincides with the expectation \eqref{eq:expPolUnif} by R\'enyi and Sulanke.
Analogously, the expected conditional variance is given by
\begin{align*}
\EE ({\sigma'}^2 | \cE)  
&= 
\EE \left( {10 \ell \over 27}\log n +{10\over 27}\sum_{i=1}^\ell \log\area(\Delta_i) + 
\frac{(10 \gamma + 2\pi^2 - 28)\ell }{27}  + O(n^{-\frac 12}) \right)
\\ &=
{10 \ell \over 27}\log n +{10\over 27}\sum_{i=1}^\ell \log F_i + 
\frac{(10 \gamma + 2\pi^2 - 48)\ell }{27}  + O(n^{-\frac 12}) 
.
\end{align*}
And for the variance of the expectation we use Corollary \ref{cor:moments-logarea} again,
\begin{align*}
\var (\mu' | \cE)
&= 
\frac 49 \var \left(  \sum_{i=1}^\ell \log \area (\Delta_i) +O(n^{-\frac 12})\ \Big|\cE \right)
\\ &= 
\frac 49  \sum_{i=1}^\ell \var \left(\log \area (\Delta_i) \ \Big|\cE \right)
+ 
\frac 89  \sum_{i=1}^\ell \cov \left( \log \area (\Delta_i) , \log \area (\Delta_{i+1}) \ \Big|\cE \right)
+ O(n^{- \frac 12 })
\\ &= 
\frac 89  \ell \left( 2- \frac{\pi^2}6 \right)
+ O(n^{-\frac 12})
\end{align*}
Hence 
\begin{align*}
\sigma^2 
&= 
\EE (\underbrace{\var  (f_0(P_{\eta})|\cE, {\bf Z})}_{= {\sigma '}^2} | \cE) 
+ \var (\underbrace{\EE (f_0(P_{\eta})|\cE, {\bf Z})}_{= \mu'} |\cE)\\
&=
{10 \ell \over 27}\log n +{10\over 27}\sum_{i=1}^\ell \log F_i + 
\frac{(10 \gamma -2 \pi^2 )\ell }{27} 
+ O(n^{-\frac 12}) .
\end{align*}
This also shows that
\begin{align}
\mu-\mu'& = {2\over 3}\sum_{i=1}^\ell \log\area(\Delta_i)+O(1),
\label{eq:moments}\\
\sigma'^2-\sigma^2 & = {10\over 27}\sum_{i=1}^\ell \log\area(\Delta_i)+O(1),
\label{eq:variance1}\\
\sigma+\sigma'& \geq \sigma = \Big({10 \ell \over 27}\log n\Big)^{\frac 12} +O(1)\label{eq:variance2}.
\end{align}

Next, we observe that 
\begin{align}\label{eq:14.05.21_2}
\nonumber &\sup_{x\in\RR} \Big| \PP(\frac {\xi-\mu}{\sigma}\leq x) - \Phi(x)\Big|\\
& \qquad\qquad= \nonumber
\sup_{y\in\RR}  \Big| \PP(\xi\leq y) - \Phi\Big(\frac {y-\mu}{\sigma}\Big)\Big| 
\\ & \qquad\qquad\leq \nonumber
\sup_{y\in \RR} \Big| \EE \PP(\xi'\leq y) - \EE \Phi\Big(\frac {y-\mu'}{\sigma'}\Big)\Big|
+
\sup_{y\in\RR}\Big |\EE  \Phi\Big(\frac {y-\mu'}{\sigma'}\Big) - \Phi\Big(\frac {y-\mu}{\sigma}\Big)\Big| \\ 
& \qquad\qquad\leq 
\EE \sup_{y\in\RR} \Big| \PP(\xi'\leq y) - \Phi\Big(\frac {y-\mu'}{\sigma'}\Big)\Big|
+
\EE \sup_{y\in\RR} \Big|\Phi\Big(\frac {y-\mu'}{\sigma'}\Big) - \Phi\Big(\frac {y-\mu}{\sigma}\Big)\Big|,
\end{align}
where the expectation is taken with respect to the law of the random vector ${\bf Z}$. From Corollary \ref{cor:PoissonWithZ} we have
\begin{equation}\label{eq:14.05.21_3}
\EE \sup_{y\in\RR} \Big| \PP(\xi'\leq y) - \Phi\Big(\frac {y-\mu'}{\sigma'}\Big)\Big|=\EE \sup_{x\in\RR} \Big| \PP\Big(\frac {\xi'-\mu'}{\sigma'}\leq x\Big) - \Phi(x)\Big|  = O\Big( {1\over\sqrt{\log n}}\Big).
\end{equation}
It remains to deal with the random variable
$$
Y_{\bf Z}:=\sup_{y\in\RR} \Big|\Phi\Big(\frac {y-\mu'}{\sigma'}\Big) - \Phi\Big(\frac {y-\mu}{\sigma}\Big)\Big|.
$$
Assume that the supremum is attained at $y_0$. 
Then, with $\phi(t):={1\over \sqrt{2\pi}}e^{-t^2/2}$ the density of standard normal distribution,
\begin{equation}\label{eq:18-05-21a}
Y_{\bf{Z}}=\Big|\Phi\Big(\frac {y_0-\mu'}{\sigma'}\Big) - \Phi\Big(\frac {y_0-\mu}{\sigma}\Big)\Big|\leq \sup_{t\in\RR}|\phi(t)|\cdot \Big|\frac {y_0-\mu'}{\sigma'} - \frac {y_0-\mu}{\sigma}\Big|\leq\Big|\frac {y_0-\mu'}{\sigma'} - \frac {y_0-\mu}{\sigma}\Big|,
\end{equation}
where $y_0$ is such that
$$
{1\over \sigma'}\phi\Big(\frac {y_0-\mu'}{\sigma'}\Big) - {1\over \sigma}\phi\Big(\frac {y_0-\mu}{\sigma}\Big)=0.
$$
By taking logarithms on both sides, we see that the last equation is equivalent to
$$
y_0^2(\sigma^2-\sigma'^2)-2y_0(\mu'\sigma^2-\mu\sigma'^2)+\mu'^2\sigma^2-\mu^2\sigma'^2-\sigma^2\sigma'^2\log\big({\sigma\over \sigma'}\big)=0.
$$
This quadratic equation has the following solutions:
$$
y^{\pm}_0={\mu'\sigma^2-\mu\sigma'^2\pm\sigma\sigma'\sqrt{(\mu'-\mu)^2+\log(\sigma/\sigma')|\sigma'^2-\sigma^2|}\over \sigma^2-\sigma'^2}.
$$
Substituting this back into \eqref{eq:18-05-21a} leads to the bound
$$
Y_{\bf{Z}}
\leq {|\mu'-\mu|+\sqrt{(\mu'-\mu)^2+\log(\sigma/\sigma')|\sigma'^2-\sigma^2|}\over \sigma+\sigma'}
\leq 
2{|\mu'-\mu|\over \sigma+\sigma'} + 
{\sqrt{|\log(\sigma/ \sigma')||\sigma'^2-\sigma^2|}\over \sigma+\sigma'},
$$
where we used the fact that $\sqrt{a+b} \leq \sqrt a + \sqrt b$ for all $a,b>0$.
Observe that given $\cE$ we have $\area(\Delta_i) \geq n^{- \frac 12}$ and thus 
$\sigma'^2 \geq {5 \ell \over 27}\log n +O(1)$. Hence, there exist constants $c_1,c_2>0$ such that $c_1<\sigma/\sigma'<c_2$. Further, $\sqrt{|\sigma'^2-\sigma^2|} \leq |\sigma'^2-\sigma^2|+1$.
Thus, using \eqref{eq:moments}, \eqref{eq:variance1}, and \eqref{eq:variance2}  we conclude that, for some constant $C_1>0$,
$$
Y_{\bf{Z}}\leq C_1\Big(\sum_{i=1}^\ell {|\log\area(\Delta_i)|\over \sqrt{\log n}}+{1\over \sqrt{\log n}}\Big),
$$
and Corollary \ref{cor:moments-logarea} yields, for another constant $C_2>0$,
$$
\EE Y_{\bf{Z}}\leq C_2\Big({\EE|\log\area(\Delta_1)|\over \sqrt{\log n}}+{1\over \sqrt{\log n}} \Big)
= {O(1)\over \sqrt{\log n}} .
$$
Together with \eqref{eq:14.05.21_2} and \eqref{eq:14.05.21_3} this completes the proof.

\end{proof}

\subsection{Step 5: Removing the condition $\cE$}

In order to remove the remaining condition $\cE$ in Lemma \ref{lem:RemovePointsZ} we use Lemma \ref{lm:transference} again. We will apply this lemma to the random variables $\xi'_n=f_0(P_{\eta}|\cE)$ and $\xi_n=f_0(P_{\eta})$. Note that condition (iv) then follows from Lemma \ref{lem:RemovePointsZ}. Checking the other conditions requires a more careful analysis.

\begin{lemma}\label{lm:condition}
The random variables $\xi_n:=f_0(P_{\eta})$ and $\xi_n':=f_0(P_{\eta}|\cE)$ satisfy conditions (i)--(iii) of Lemma \ref{lm:transference} with $\varepsilon_1(n)=c_1(\log n)^{3/2}n^{-1/4}$, $\varepsilon_2(n)=c_2(\log n)^3n^{-1/4}$, $\varepsilon_3(n)=c_3n^{-1/4}$, where $c_1,c_2,c_3>0$ are positive constants not depending on $n$.
\end{lemma}
\begin{proof}
The proof of this lemma will basically follow the lines of the proof of Lemma 8.2 in \cite{BR10}. We will start by estimating $\EE(f_0(P_{\eta})^k|\overline{\cE})$ for $k=1,2$. For this we assume additionally that the event $\cA^\pi_n \cap \cB^\pi_n$ holds, that is 
$$P(v\ge b_0 n^{-1} \log n )\subset P_{\eta}\qquad\text{and}\qquad \ \#(\eta \cap P(v< b_0 n^{-1} \log n ))\leq c_0 (\log n)^2 ,$$ 
and make use of Lemma \ref{le:floating-Peta}. It follows that 
$
f_0(P_{\eta})^k \ind (\cA^\pi_n \cap \cB^\pi_n) \leq 
c_0^{k} (\log n)^{2k}
. $
As a consequence,
\begin{align*}
\EE(f_0(P_{\eta})^k|\overline{\cE})&=\EE(f_0(P_{\eta})^k \ind (\overline{\cA^\pi_n \cap \cB^\pi_n})|\overline{\cE})+ \EE(f_0(P_{\eta})^k\ind (\cA^\pi_n \cap \cB^\pi_n)|\overline{\cE})
\\ &\leq
\frac{\EE((\# \eta )^k \ind (\overline{\cA^\pi_n \cap \cB^\pi_n}) \ind (\overline{\cE}))}{\PP(\overline{\cE})}+ C\,(\log n)^{2k}.
\\ &\leq
\frac{(\EE(\# \eta )^{2k})^{\frac12} \PP (\overline{\cA^\pi_n \cap \cB^\pi_n})^{\frac 12}}{\PP(\overline{\cE})}+ C\,(\log n)^{2k},
\end{align*}
by H\"older's inequality and where $C>0$ is some constant. Because for $m \geq 1$, $\EE N^{m} \leq m^m (n^m+1)$ for a Poisson random variable $N$ with mean $np$, 
$$
(\EE(\# \eta )^{2k})^{\frac12}  = O(n^{k}),
$$
and from Corollary \ref{cor:complcE} and Lemma \ref{le:floating-Peta} we see that
$$
{\PP(\overline{\cA^\pi_n \cap \cB^\pi_n})^{\frac 12}\over \PP(\overline{\cE})} = O( n^{- \frac 52}).
$$
Thus, for $k=1,2$
\begin{equation}\label{eq:3}
\EE(f_0(P_{\eta})^k|\overline{\cE}) = O( (\log n)^{2k}).
\end{equation}

In order to verify conditions (i)--(iii) in Lemma \ref{lm:transference} we will use the following simple inequality from \cite[Claim 8.3]{BR10}, which says that
$$
|\EE(\zeta)-\EE(\zeta|A)|\leq (\EE(\zeta|A)+\EE(\zeta|\overline{A}))\PP(\overline{A}), 
$$
for any non-negative random variable $\zeta$ and any event $A$.

For condition (i) we take $\zeta = \EE f_0(P_{\eta})$ and $A=\cE$. 
By Lemma \ref{lem:RemovePointsZ} we get
$ \EE (f_0(P_{\eta})|\cE) = O( \log n)$.
Using this and \eqref{eq:3}, and the estimate in Corollary \ref{cor:complcE} for $\PP (\overline{\cE})$, we conclude that
\begin{equation}\label{eq:4}
|\EE(f_0(P_{\eta}))-\EE(f_0(P_{\eta})|\cE)| = O((\log n)^2n^{-1/4}),
\end{equation}
and, thus,
\begin{equation}\label{eq:EE-Poiss}
\EE f_0(P_{\eta}) = 
{2 \ell \over 3}\log n + \frac 23 \sum_{i=1}^\ell \log F_i  + \frac{2 \gamma \ell }3  + O((\log n)^2n^{-1/4}) .
\end{equation}
In the same way, for condition (ii) we take $\zeta =\EE f_0(P_{\eta})^2$ and $A=\cE$. Then
\begin{align*}
|\var(f_0(P_{\eta}))-\var(f_0(P_{\eta})|\cE)|
&\leq |\EE f_0(P_{\eta})^2-\EE(f_0(P_{\eta})^2|\cE)|+|(\EE f_0(P_{\eta}))^2-(\EE(f_0(P_{\eta})|\cE))^2|.
\end{align*}
For the first term we get from Lemma \ref{lem:RemovePointsZ} the bound
$$
\EE(f_0(P_{\eta})^2|\cE)
=
\var (f_0(P_{\eta})|\cE)
 + (\EE (f_0(P_{\eta})|\cE))^2
= O( (\log n)^2),
$$
and combine this with \eqref{eq:3} and Corollary \ref{cor:complcE}, in order to obtain
\begin{align*}
|\EE f_0(P_{\eta})^2-\EE(f_0(P_{\eta})^2|\cE)|&\leq (\EE(f_0(P_{\eta})^2|\cE)+\EE(f_0(P_{\eta})^2|\overline{\cE}))\PP(\overline{\cE})\\
& = O((\log n)^4n^{-1/4}).
\end{align*}
For the second term we note that again by Lemma \ref{lem:RemovePointsZ} and \eqref{eq:4} we get
\begin{align*}
|(\EE f_0(P_{\eta}))^2-(\EE(f_0(P_{\eta})|\cE))^2|
& = O( (\log n)^3n^{-1/4}).
\end{align*}
Putting these estimates together we conclude that
$$
|\var(f_0(P_{\eta}))-\var(f_0(P_{\eta})|\cE)| = O( (\log n)^4n^{-1/4}),
$$
implying that 
\begin{equation}\label{eq:Var-Poiss}
\var(f_0(P_{\eta})) =
{10 \ell \over 27}\log n +{10\over 27}\sum_{i=1}^\ell \log F_i + 
\frac{(10 \gamma -2 \pi^2 )\ell }{27} 
+ O( (\log n)^4n^{-1/4}) . 
\end{equation}
This shows that $\varepsilon_2(n)=c_2(\log n)^3n^{-1/4}$, where $c_2>0$ is some absolute constant.
Analogously, by \eqref{eq:4} we set $\varepsilon_1(n)=c_1(\log n)^{3/2}n^{-1/4}$, where $c_1>0$ is some absolute constant.

Finally, taking $\zeta = \ind (f_0(P_{\eta})\leq x)$ and $A=\cE$ we obtain from Corollary \ref{cor:complcE} that
$$
|\PP(f_0(P_{\eta}|\cE)\leq x)-\PP(f_0(P_{\eta})\leq x)|\leq 2\PP(\overline{\cE}) = O( n^{-1/4}).
$$
This completes the argument.
\end{proof}

By Lemma \ref{lm:transference} and Lemma \ref{lm:condition} together with Lemma \ref{lem:RemovePointsZ} we conclude the required Berry-Esseen bound for the Poisson model of random polygons. Expectation and variance have been obtained in \eqref{eq:EE-Poiss} and \eqref{eq:Var-Poiss}.

\begin{theorem}\label{thm:BerryEsseenPoisson}
Consider the Poisson random polygon $P_\eta$ induced by a homogeneous Poisson point process $\eta$ in a polygon $P$ of unit area with $\EE(\# \eta)=n$. Then, for any $n\geq 2$,
$$
\sup_{x\in \RR}\Big|\PP\Big({f_0(P_{\eta})-\EE f_0(P_{\eta})\over \sqrt{\var f_0(P_{\eta})}}\ge x\Big)-\Phi(x)\Big|\leq {c\over \sqrt{\log n}}
$$
for some constant $c>0$ independent of $n$, with
\begin{align*}
\EE f_0(P_{\eta}) &= 
{2 \ell \over 3}\log n + \frac 23 \sum_{i=1}^\ell \log F_i  + \frac{2 \gamma \ell }3  + O((\log n)^2n^{-1/4}) 
\intertext{and}
\var(f_0(P_{\eta})) &=
{10 \ell \over 27}\log n +{10\over 27}\sum_{i=1}^\ell \log F_i + 
\frac{(10 \gamma -2 \pi^2 )\ell }{27} 
+ O( (\log n)^4n^{-1/4}) .
\end{align*}
\end{theorem}

\subsection{Step 6: Going back to uniform model}

This is the last step in the proof of Theorem \ref{thm:main}, in which we apply Lemma \ref{lm:transference} with $\xi'_n=f_0(P_{\eta})$ and $\xi_n=f_0(P_n)$. Condition (iv) there holds due to Theorem \ref{thm:BerryEsseenPoisson} with $\varepsilon_4(n)=c_4/\sqrt{\log n}$ for some $c_4>0$. Condition (i) with $\varepsilon_1(n)=c_1/\sqrt{\log n}$, $c_1>0$, follows from \eqref{eq:expPolUnif} and Theorem \ref{thm:BerryEsseenPoisson}. 

In order to check condition (ii) we need a more precise asymptotics for the variance of $f_0(P_n)$ compared to the one given by \eqref{eq:varPolUnif}. In particular we need a result of the form $\var(f_0(P_{n}))={10 \ell\over 27} \log n +O(1)$. 
First note that by the trivial estimate $f_0(P_\eta) \leq \# \eta$ and using Lemma \ref{le:floating-Peta} we have
\begin{align} \label{eq:diff-expP_eta-cA} 
| \EE (f_0(P_\eta)^k \ind(\cA^\pi_n)) - \EE f_0(P_\eta)^k |
&= \nonumber 
| \EE ( f_0(P_\eta)^k ( \ind (\cA^\pi_n)  -1 )) |
\\ \nonumber &\leq
(\EE  (f_0(P_\eta)^{2k})^{\frac 12}  (\EE ( \ind (\cA^\pi_n) -1  )^2 )^{\frac 12}
\\ & = O(
n^k \, \PP( \overline{\cA^\pi_n})^{\frac 12})
= O(
n^{k-3})
\end{align}
for $k=1,2$, say.
The definition of the variance gives 
\begin{align*}
| \var (f_0(P_\eta) \ind (\cA^\pi_n)) & - \var f_0(P_\eta) |
\\ & \leq
| \EE (f_0(P_\eta)^2 \ind(\cA^\pi_n)) - \EE f_0(P_\eta)^2 | 
+ | (\EE (f_0(P_\eta) \ind( \cA^\pi_n)))^2 - (\EE f_0(P_\eta))^2 |
.
\end{align*}
We use now \eqref{eq:diff-expP_eta-cA} to bound the first term by $n^{-1}$. Theorem \ref{thm:BerryEsseenPoisson} and \eqref{eq:diff-expP_eta-cA} shows, that 
\begin{align*}
| (\EE (f_0(P_\eta) \ind(\cA^\pi_n)))^2 - (\EE f_0(P_\eta))^2 |
=
| \EE (f_0(P_\eta) \ind( \cA^\pi_n)) - \EE f_0(P_\eta) |
| \EE (f_0(P_\eta) \ind(\cA^\pi_n)) + \EE f_0(P_\eta)  |
\end{align*}
is bounded by $n^{-2} \log n  = O( n^{-1})$. 
Hence 
\begin{align*}
| \var (f_0(P_\eta) \ind(\cA^\pi_n)) - \var f_0(P_\eta) |
& = O(
n^{-1}).
\end{align*}
The same result holds for 
$ \var (f_0(P_n) \ind(\cA_n)) - \var f_0(P_n) $ with an actually  slightly simpler proof because $f_0(P_n) \leq n$ in which $n$ is non-random.
These estimates show that
\begin{align*}
| \var f_0(P_\eta) - \var f_0(P_n) |
 &\leq 
|  \var (f_0(P_\eta) \ind(\cA^\pi_n)) - \var (f_0(P_n) \ind(\cA_n)) | + O(n^{-1})
\\ & \leq 
|  \EE (f_0(P_\eta) \ind(\cA^\pi_n))^2 - \EE (f_0(P_n) \ind(\cA_n))^2 | 
\\ &\qquad + 
| \EE (f_0(P_\eta) \ind(\cA^\pi_n)) - \EE (f_0(P_n) \ind(\cA_n)) |\,  \left( \EE f_0(P_\eta) + \EE f_0(P_n) \right) 
 + O(n^{-1})
.
\end{align*}
For $P_\eta$ we have that $\# \eta \cap P(v\ge b_0 n^{-1} \log n)$ is Poisson distributed with mean
$$ 
np 
:= n\, \area(P(v\ge b_0 n^{-1} \log n)) 
= O(  (\log n)^2 ) ,
$$
by \eqref{eq:wetpart}, and for $P_n$ the number of points in $P(v\ge b_0 n^{-1} \log n)$ is binomial distributed with mean $p$. Denote by $E_m$ the event that precisely $m$ points of the Poisson or binomial process are in $P(v\ge b_0 n^{-1} \log n)$.
Coupling both processes in the canonical way yields
\begin{align*}
|  \EE (f_0(P_\eta) \ind(\cA^\pi_n))^k   &   - \EE (f_0(P_n) \ind(\cA_n))^k | 
\\ & = 
\sum_{m=0}^\infty \EE (f_0(P_n) \ind(\cA_n))^k | E_m) \left|\frac {(np)^m}{m!} e^{-np} - {n \choose m} p^m (1-p)^{n-m} \right|
\\ & \leq 
\sum_{m=0}^\infty m^k \left|\frac {(np)^m}{m!} e^{-np} - {n \choose m} p^m (1-p)^{n-m} \right| .
\end{align*}
Lemma \ref{le:diff-Poisson-binom} and the fact that the expected number of vertices in both models is of order $\log n$ implies
\begin{align*}
| \var f_0(P_\eta) - \var f_0(P_n) |
& = O(
n^{-1}  (\log n)^5 )
.
\end{align*}
Thus,
$$
\var f_0(P_n) 
=
{10 \ell \over 27}\log n +{10\over 27}\sum_{i=1}^\ell \log F_i + 
\frac{(10 \gamma -2 \pi^2 )\ell }{27} 
+ O( (\log n)^4n^{-1/4}) ,
$$
which proves the variance expansion in Theorem \eqref{thm:main} and shows that condition (ii) holds with $\varepsilon_2(n)=c_2 n^{-1} (\log n)^4$ for some $c_2>0$ independent of $n$.

The slightly better estimate for the expectation in Theorem \ref{thm:main} (in comparison to \eqref{eq:expPolUnif}) follows analogously from Theorem \ref{thm:BerryEsseenPoisson} and the estimate
\begin{align*}
| \EE f_0(P_\eta) - \EE f_0(P_n) |
 &\leq 
|  \EE (f_0(P_\eta) \ind(\cA^\pi_n)) - \EE (f_0(P_n) \ind(\cA_n)) | + O(n^{-1})
=
O(n^{-1}(\log n)^4)
.
\end{align*}

Condition (iii) can be verified following the same approach, which was already used in Lemma \ref{lm:BerryEssenPoissonChain}. By Lemma \ref{le:floating-Pn}, Lemma \ref{le:floating-Peta} and Lemma \ref{le:diff-Poisson-binom},
\begin{align*}
\big|\PP(f_0(P_n)\leq x)-\PP(f_0(P_{\eta})\leq x)\big|
& = 
\big|\PP(f_0(P_n)\leq x, \cA_n)-\PP(f_0(P_{\eta})\leq x, \cA_n^\pi)\big| + O(n^{-6})
\\ & \leq
\sum_{m=0}^\infty \left| \frac{(np)^m}{m!} e^{- np} - {n \choose m} p^m (1-p)^{n-m} \right| + O(n^{-6})
\\ & \leq
2p + O(n^{-6})\\
& = O(
n^{-1} (\log n)^2).
\end{align*}
Thus, (iii) holds with $\varepsilon_3(n)=c_3 n^{-1} (\log n)^2$ for some constant $c_3>0$ independent of $n$. Finally, by combining all estimates and using Lemma \ref{lm:transference} we complete the proof of Theorem \ref{thm:main}.\qed

\subsection*{Acknowledgement}
We would like to thank Sam Johnston (Bath) for an enlightening discussion on the content of Lemma \ref{lm:GlueBerryEsseen}. This work has been initiated during the virtual Hausdorff Trimester Program \textit{The Interplay between High Dimensional Geometry and Probability}. \\ CT was supported by the German Research Foundation (DFG) via SPP 2265 \textit{Random Geometric Systems}. AG  was supported by the DFG under Germany's Excellence Strategy  EXC 2044 -- 390685587, \textit{Mathematics M\"unster: Dynamics - Geometry - Structure}.

\addcontentsline{toc}{section}{References}

\bibliographystyle{acm}

\end{document}